\documentclass{amsart}
\usepackage[T1]{fontenc}
\usepackage[utf8]{inputenc}%pour les accents, sinon latin9
\usepackage{amsmath,amsthm,amsfonts,amscd,amssymb,eucal,latexsym,mathrsfs}
\usepackage{stmaryrd}
\usepackage{enumerate}
\usepackage[hidelinks]{hyperref}
\usepackage[all]{xy}
\usepackage{etoolbox}

\usepackage{tikz,tikz-cd}
\usetikzlibrary{shapes.geometric}
\usetikzlibrary{arrows}
\usepackage{tqft}

\usetikzlibrary{positioning}
\usepackage{float}
\usepackage{MnSymbol}
\usetikzlibrary{matrix}

\newcommand{\toc}{\tableofcontents}

\theoremstyle{plain}
\newtheorem{theorem}{Theorem}[section]
\newtheorem*{theorem*}{Theorem}

\newtheorem*{corollary*}{Corollary}

\newtheorem{lemma}[theorem]{Lemma}

\newtheorem{proposition}[theorem]{Proposition}

\theoremstyle{definition}
\newtheorem{remark}[theorem]{Remark}

\newtheorem{example}[theorem]{Example}

\newtheorem{definition}[theorem]{Definition}
\newtheorem*{definition*}{Definition}

\usetikzlibrary{calc,graphs}
\usepackage{xcolor}
\usetikzlibrary{arrows,decorations.pathmorphing}

\DeclareMathOperator{\Lk}{\mathrm{Lk}}
\DeclareMathOperator{\Ner}{\mathrm{Ner}}
\DeclareMathOperator{\Co}{\mathrm{Co}}

\newcommand{\acts}{\curvearrowright}

\newcommand{\G}{\Gamma}
\newcommand{\p}{\varphi}

\renewcommand{\d}{\mathrm{d}}

\newcommand{\IR}{\mathbb{R}}

\newcommand{\CI}{\mathbb{C}}
\newcommand{\IF}{\mathbb{F}}
\newcommand{\IZ}{\mathbb{Z}}
\newcommand{\FI}{\mathbb{F}}
\newcommand{\ZI}{\mathbb{Z}}
\newcommand{\IN}{\mathbb{N}}

\newcommand{\SI}{\mathbb{S}}

\newcommand{\ft}{\text{ft}}

\DeclareMathOperator{\ran}{\mathrm{ran}}

\DeclareMathOperator{\dom}{\mathrm{dom}}

\DeclareMathOperator{\del}{\partial}

\DeclareMathOperator{\Aut}{\mathrm{Aut}}

\DeclareMathOperator{\ssi}{\Leftrightarrow}

\DeclareMathOperator{\inj}{\hookrightarrow}

\DeclareMathOperator{\lra}{\leftrightarrow}

\DeclareMathOperator{\surj}{\twoheadrightarrow}

\DeclareMathOperator{\Hom}{\mathrm{Hom}}

\DeclareMathOperator{\spn}{\mathrm{span}}

\newcommand{\MK}{\mathrm{MK}}

\newcommand{\Bord}{\mathrm{Bord}}

\DeclareMathOperator{\rk}{\mathrm{rk}}

\let\Col\relax
\DeclareMathOperator{\Col}{\mathrm{Col}}

\newcommand{\sq}{$\frac 7 4$}

\newcommand{\ts}{\textsection}

\newcommand{\metr}{\mathrm{metr}}
\newcommand{\simp}{\mathrm{simp}}

\newcommand{\ip}[1]{\langle#1\rangle} % \ip{a,b} gives us <a,b>

\author{Sylvain Barr\'e}
\author{Mika\"el Pichot}
\title{Surgery on discrete groups}
\begin{document}

\begin{abstract}
We study constructions of groups, in particular of groups of intermediate rank, which are accessible to surgery techniques. 
\end{abstract}

\maketitle

\section{Introduction}

The present paper makes use of some fragments of the cobordism theory to study  countable discrete groups, especially from the point of view of geometric group theory.
The aim of the paper is to study a notion of \emph{group cobordisms}, which fit into a \emph{cobordism category} whose constituants  are 2-dimensional in nature.

Classical surgery is a topological operation on manifolds. It  has had great achievements, starting with the construction of exotic differentiable structures on the sphere of dimension 7. 
  We are looking to use surgery techniques to construct  countable  groups. A concrete such construction was given in \cite{bbelge}.
  
The classical category $\Bord_n$ of $n$-cobordisms is easy to describe in dimension $n=2$ (it is generated, as a monoidal category, by caps, pair of pants, the cylinder, and the twist pair of  cylinders, which are cobordisms between unions of circles). In the case of groups, the 2-dimensional spaces that will serve as group cobordisms are branching simplicial complexes. In essence, the dimension is $n=4$, since every finitely presented group can be written as the fundamental group of a closed manifold of dimension 4. The theory  however is not captured by $\Bord_4$. We are looking in particular for results that can be applied directly to the study of groups. The ``ST lemma'' formulated in \ts\ref{S - ST Lemma}, or Theorem \ref{T - cobordism 74}, are concrete examples of such results.   We note that  surgery theory is expected to see further developments in the direction of coarse geometry (compare \cite[\ts 4]{rosenberg2014surgery} and the references therein).

The group cobordisms are the arrows of a category whose objects are called \emph{collars}. Collars are certain 2-dimensional complexes defined in \ts\ref{S - Collars}, and are reminiscent of the usual notion of collar for Riemann surfaces. We will distinguish between metric and simplicial collars, and group cobordisms.
The ST lemma gives information (under some assumptions) on the smallest collars that can be used to construct groups of nonpositive curvature.

The main applications that we have in mind in this paper are to the construction of  groups of intermediate rank (we refer for example  to \cite{chambers} for a brief introduction to these groups). For example, we use in \ts\ref{S - type preserving constructions} the two group cobordisms described by Theorem  \ref{T - cobordism 74}  to construct infinitely many  groups of Moebius--Kantor type, in the sense of \cite{rd}. 
This can be achieved because the surgeries are performed in a controlled way. In particular, they preserve the ``type'' of the group under consideration; for instance, a surgery operation starting with a group of Moebius--Kantor type  indeed returns a group of Moebius--Kantor type . (In \cite{bbelge}, the type $\tilde A_2$ of the group is not preserved by surgery---in fact, the surgery construction explained there transforms a group with Kazhdan's property T into a group with the Haagerup property.)

The ``type'' can be defined in a simple way by listing the constraints on  the local data that define the given class of groups (see \ts\ref{S - types} for the definition---a type can be metric, or simplicial); to such a type $A$ is associated a category $\Bord_A$ of group cobordisms of type $A$, which describes the possible surgeries for the groups of type $A$. The categories $\Bord_A$ are our main objects of study.  
For the purposes of rank interpolation, the case study is the category $\Bord_{\MK}$  describing the surgeries for the groups of Moebius--Kantor type.   
We will obtain some partial information on this category in this paper.

  In the course of exploring the subject (and  fitting it to the requirements of rank interpolation),  a few concepts and facts have emerged, some of which we  mention now: a) in \ts\ref{S - model geometry}, we define a notion of \emph{model geometry}, which are the models, or building blocks, that can be used to construct any group of a given type. They will provide the simplest group cobordisms in $\Bord_A$;  b) in  \ts\ref{S - 2/3 transitive}, we discuss a notion \emph{$\frac 2 3$-transitivity} for group actions on simplicial 2-complexes. An action of a group on a simplicial 2-complex is said to be $\frac 2 3$-transitive if every triangle intersects  at most two orbits of vertices. These actions provide the simplest examples of collars; c)   the ``ST lemma'', mentioned above, is a classification of the smallest collars that can connect two model geometries of nonpositive curvature. This is a useful result in particular for the the study of groups of intermediate rank, which act on 2-complexes of nonpositive curvature  (in the CAT(0) sense); d) double covers are a useful source of explicit collars, cf.\ \ts\ref{S-double cover}. e)  surgery on  collars themselves also leads to rather exotic groups, namely the  fake double covers described in \ts\ref{S-fake cover}; f) in \ts \ref {S - 158}, we construct a group of  mixed type Moebius--Kantor and $\tilde A_2$, which illustrates constructibility questions about types raised in  \ts\ref{S - types}; g) the main result on group cobordisms is Theorem \ref{T - cobordism 74}, which concerns non filling cobordisms of type $A_\MK$, and is used to construct the new groups of Moebius--Kantor type as mentioned above (explicit drawings of group cobordisms can be found in Figures \ref{F - Cobordism 32} and \ref{F - Cobordism 2}).

The category $\Bord_A$ is a global object associated with the type $A$; a different global object,  which also contains important information about the type $A$, is introduced in \ts\ref{S - types}. It is a dynamical system called the space of complexes of type $A$, and it is a generalization (in particular to complexes of intermediate rank) of the space of triangle buildings defined in \cite{E1}, where the type $A$ was the (Coxeter) type $\tilde A_2$. A notion of ``indicability'' (existence of a surjective morphism to $\ZI$) in this space is discussed in \ts \ref{S - Global Indicability} in relation with the existence of sufficiently many $\frac 2 3 $-transitive actions for groups of type $A$. This provides finiteness information on categories such as $\Bord_{\tilde A_2}$, for which Kazhdan's property T can be used (see Theorem \ref{T - buildings 23}).  In the opposite direction, the  surgery constructions in the category $\Bord_{\MK}$ shed light on some dynamical properties of this space for Moebius--Kantor complexes. This latter point is discussed further in \ts \ref{S - Results} below, and the main results are proved in \ts \ref{S - type preserving constructions}.

The interactions between geometric group theory and category theory look promising. The present paper raises a few  general questions to which we hope to return elsewhere---for example, it would be desirable for us to a) have a deeper understanding of the categories $\Bord_A$ associated with a type $A$ (including $\Bord_{\MK}$), and b)  study  quantum invariants for groups of intermediate rank (including groups of Moebius--Kantor type) arising from topological quantum field theory  constructions over group cobordism categories.

\section{Structure of the paper and main results}\label{S - Results}

The first  objective of the paper is to define the category $\Bord_A$ of group cobordisms of type $A$.

The objects of $\Bord_A$, the collars, are topological closures of product spaces (of a graph with an open segment) inside spaces of type $A$. We describe carefully which topological closures we admit as collars in \ts\ref{S - Collars}.

Recall from the introduction that an action of a group on a simplicial 2-complex is  $\frac 2 3$-transitive if every triangle intersects at most two orbits of vertices.
Using collars and surgery we will prove:

\begin{theorem*}[See Theorem \ref{T - 23 transitive}]
  Let $A$ be a finite combinatorial type. Assume that there exist infinitely many pairwise nonconjugate free $\frac 2 3$-transitive actions  with compact quotients on simplicial complexes of type $A$.  Then there exists an indicable group of type $A$.
\end{theorem*}

This result, combined with Kazhdan's property T, implies for example that:

 \begin{corollary*}[See Theorem \ref{T - buildings 23}]
The family of groups which admit a free $\frac 2 3$-transitive action on a  Bruhat-Tits building of type $\tilde A_2$ and order $q$ is finite. 
\end{corollary*}

A more detailed study of collars is done in \ts\ref{S - ST Lemma}, where the ST lemma is proved and the span decomposition of the resulting collars is described. 

The morphisms of the category $\Bord_A$ are defined in \ts\ref{S - cobordisms} as follows:

\begin{definition*}
A \emph{group cobordism} is a 2-complex $X$ together with a pair $(C,D)$ of collars in $\Col(X)$ whose boundaries $\del^-C$ and $\del^+ D$ form a partition of the topological boundary of $X$:
\[
\del X= \del^-C\sqcup \del ^+ D.
\]
\end{definition*} 

The category $\Bord_A$ describes the possible constructions by surgery of 2-complexes (and groups) of type $A$. 

The second main objective of the paper is to begin the study these cobordism categories, starting in particular with the simplest nontrivial case of spaces of rank \sq.  We  prove that:

\begin{theorem*}[See Theorem \ref{T - cobordism 74}]
There exist precisely two non filling group cobordisms of rank \sq\ with one vertex whose boundary collars are connected of type $ST$. Furthermore, the two cobordisms are self-dual, and their collars are pairwise isomorphic,  self-dual, and of type $S$.
\end{theorem*}

This result can be used to construct infinitely many groups of Moebius--Kantor type. (Finitely many examples of such groups, up to finite index, are presented in \cite{rd}.) It also shows that the category $\Bord_{\MK}$ is infinite, which contrasts with the finiteness results of Theorem \ref{T - buildings 23}.

The construction of groups from the above theorem is described in \ts\ref{S - type preserving constructions}. As mentioned in the introduction, these constructions can also be applied to the study the dynamical system of Moebius--Kantor complexes. We now explain these results in more details. 
 
If $A$ is a type, then the space $\Lambda_{A}$ (of pointed isomorphism classes) of complexes of type $A$ is a dynamical system. It is endowed with the so-called  ``base point dynamic'' is defined by the equivalence relation $R_A\subset \Lambda_A\times \Lambda_A$, given by 
\[
(X,*)\sim_{R_A} (X',*') \ssi \exists \p\colon X\stackrel{\sim}\to X'\text{ isometry}.
\]  

The study of Euclidean buildings (of rank 2) from the point of view of dynamical systems was begun in \cite{E1,E2,Esurv}, and the present definition of $\Lambda_A$ extends these considerations to more general types. 

An open problem raised in these papers (see e.g.\ Question 7.2 in \cite{Esurv}) was whether the space $\Lambda_A$, in the case $A=\tilde A_2$ of Euclidean triangle buildings, supports a diffuse invariant measure. The original motivation was that if $\Lambda_A$ indeed admits such a measure, then it provides a new source of probability measure preserving (pmp) equivalence relations with the property T of Kazhdan---one which involves no group in the constructions. Observe in particular that the leaves defining $\Lambda_A$ are pairwise non isomorphic by definition of the base point dynamics.

It is desirable to formulate the absence of such a group in a precise way, and this leads to  \cite[Question 7.3]{Esurv}. Is there---in case an invariant measure exists---a nontrivial equivalence subrelation of $R_A$ which is the orbit partition of an essentially free action of some discrete countable group? A hypothetical such $G$ would be a ``structure group'' for the complexes of type $A$, which allows to construct pairwise non isomorphic (quasi-periodic) such complexes (namely, Euclidean buildings) by perturbing the geometry ``above $G$'' using it as a blueprint. 

The first example of a pmp equivalence relation that cannot be written as the orbit partition of an essentially free action of a countable group was obtained by Furman \cite[Theorem D]{furman1999orbit}.

The surgery constructions described above shed  light on these questions for groups of intermediate rank---of type $A_\MK $---instead of rank 2---of type $\tilde A_2$.

\begin{theorem*}[See Theorem \ref{T - 74 diffuse}]
The space $\Lambda_\MK $ supports a diffuse invariant probability measure. 
\end{theorem*}

In fact, the  dihedral group $D_\infty$ is  a structure group for the type $A_\MK $ (see Theorem \ref{T - structure group}).

\bigskip
\bigskip

\textbf{Acknowledgment.} The second author was partially supported by an NSERC discovery grant. He is indebted to Yasuyuki Kawahigashi for an invitation to the University of Tokyo in Spring 2016, where part of this work was carried out.

\toc

\section{Model geometries}\label{S - model geometry}

By 2-complex we mean a polyhedral CW-complex of dimension 2.

\begin{definition}
A  \emph{model geometry} (of dimension 2) is a connected 2-complex $M$ with a distinguished vertex (the center) satisfying the following two conditions:
\begin{enumerate}
\item every vertex in $M$ distinct from the center is adjacent to it
\item every loop  is attached to the center
\end{enumerate}
\end{definition}

Let $M$ be a model geometry.
There are two types of edges in $M$ besides the loops:
\begin{enumerate}
\item[-] the \emph{half edges}, exactly one of whose extremities coincide with the center
\item[-] the \emph{boundary edges}, none of whose extremities coincide with the center
\end{enumerate}
We always assume that the half-edges are simple, and that every boundary edge is contained in a unique face of $M$.

The \emph{core}  of $M$ is the maximal subcomplex $\Co M$ of $M$ whose vertex set contains only the center. It contains the loops and the faces attached (exclusively) to loops. Thus we have an homotopy equivalence, 
\[
M\sim \Co M,
\] 
using (2), since the boundary edges are contained in a unique face. The fundamental group $\pi_1(M)\simeq\pi_1(\Co M)$ is called the model group associated with $M$.

The \emph{link}  of $M$ is the link $\Lk M$ of its center. (Recall that the link at a point  in a 2-complex is the sphere of small radius around that point.)   

The \emph{boundary} $\del M$ of $M$ has as vertex set the vertices of $M$ distinct from the center, and as edge set the boundary edges.

 The vertex set of $\del M$ is a subset of the set of vertices of the link:
 \[
 \del^0 M\subset  \Lk^0 M.
 \]

To every edge $e$ in $\del M$ is attached a weight $w(e)$ defined by:
\begin{align*}
w(e)&:=1+\text{number of inner loops of the (unique) face $f$ containing $e$}\\  
&=|f| -2,\text{ where }|f|:=\text{number of edges of }f
\end{align*}
The weight represents the ``inner perimeter'' of the face $f$. (The addition of 1 in the definition accounts for the two inner half edges that are attached to $e$.) 

\bigskip

The simplest examples of model geometries are cones on graphs.  

\begin{example}
In the Euclidean plane tessellated by   squares, the model geometry at every vertex is a flat square (with dashed boundary in the drawing)

\begin{center}
\includegraphics[width=3cm]{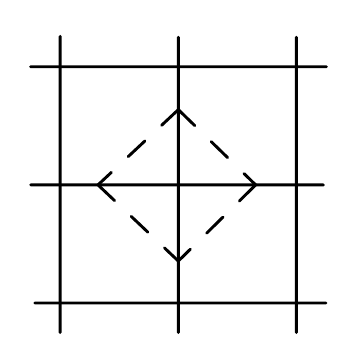} 
\end{center}
with 4 half edges and 4 boundary edges. 
\end{example}

Model geometries provide an explicit source of group cobordisms, and will be used as such later in the paper.  

\bigskip

We say that a face in a 2-complex $X$ is \emph{crossing} if it contains at least 2 distinct vertices of $X$. Every 2-complex can be decomposed into model geometries and crossing faces as follows.

Let $r$ be a crossing face in $X$. The vertex set of  $r$ can be partitioned according to the equivalence relation generated by the relation 

\begin{center}
``being adjacent in $r$ and corresponding to the same vertex of $X$''. 
\end{center}

For example, in the following figure

\begin{center}
\includegraphics[width=5cm]{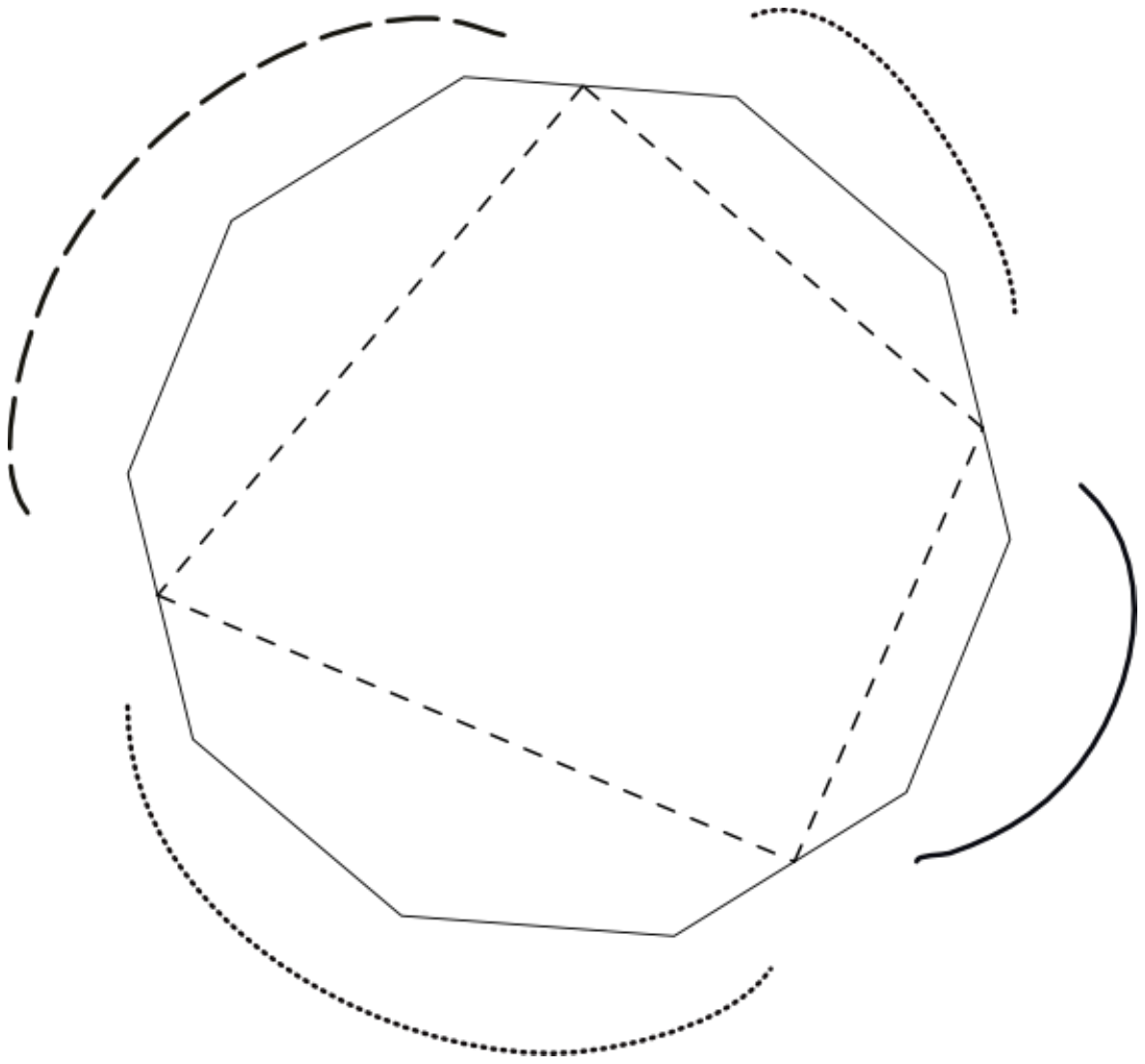} 
\end{center}

\noindent we have a face with 10 vertices and four classes of vertices, and two of these classes (represented with dots) belonging to the same vertex of $X$.

The ``middle points'' of every edge $e$ of $r$ whose extremities are inequivalent form the vertex set of a polygon (with filled interior) inscribed in $r$, which is represented with a dashed boundary in the figure. 

Assume that $X$ is finite. Let $r_1,\ldots, r_n$ denote the crossing faces of $X$ and $s_1,\ldots, s_n$ denote the respective inscribed polygons. 

\begin{definition}\label{D - model geometry} The model geometry of $X$ at $x$ is the 2-complex
\[
M_x = \bigcup_{i=1}^n (r_i\setminus s_i)_{|x}
\] 
where $r_i\setminus s_i$ denotes the pieces outside $s_i$ and $_{|x}$ selects the components corresponding to $x$ (which may not be connected). The fundamental group 
\[
G_x:=\pi_1(\Co M_x)\simeq\pi_1(M_x)
\] 
is called the model group of $X$ at $x$. 
\end{definition}

The 2-complex $X$ can be reconstructed in an obvious way as a quotient 
\[
X \simeq\left (\bigsqcup_{x\in X^0} M_x \sqcup \bigsqcup_{i=1}^n s_i  \right) /\sim
\]
where $s_i$ is a (filled) polygon inscribed in the crossing relation $r_i$, possibly reduced to a segment, and the relation $\sim$ attaches the pieces  
 together along their boundaries.

Furthermore, the following combinatorial relation holds:
\[
\sum_{r\in X^2} |r|=\sum_{x\in X}^k\sum_{e\in \del M_x} w(e)\ \ \ \  \ \ \ \text{(weight equation)}.
\]
(This corresponds to two ways of counting the number of edges of the crossing faces in $X$.)

Note that in the decomposition
\[
X \simeq\left (\bigsqcup_{x\in X^0} M_x \sqcup \bigsqcup_{i=1}^n s_i  \right) /\sim
\]
the quotient 2-complex structure on the right is a subdivision of the 2-complex structure of $X$. We call the map $M_x\inj X$ a \emph{germ embedding} of $M_x$ in $X$. It is not, strictly speaking, a 2-complex complex map, but can be seen as one once $X$ is endowed with the quotient structure coming from the quotient structure on the right.

The above discussion holds for infinite complexes, if one includes infinite families of model geometries and shapes, and leaves out the weight equation.

We are especially interested, in particular for the purpose of rank interpolation, in metric versions of this decomposition.

By metric 2-complex we mean that the edges and faces are endowed with a compatible metric (and the attaching maps are isometries), and  we shall assume (for simplicity) that the faces are flat, by which we  mean isometric to a Euclidean disk with a polygonal boundary, which is strictly convex in the sense that the inner angle at every vertex is $<\pi$. We say that a metric 2-complex is nonpositively curved if it satisfies the link condition  condition \cite{gromov1996geometric} (which is to say that the girth of its links is $\geq 2\pi$ with respect to the angular metric).

If $X$ is a metric 2-complex, the model geometry $M_x$ at a vertex $x$ of $X$, in the sense of Definition \ref{D - model geometry},  is endowed with the induced metric, where by definition the half-edges extend up to the middle points of edges, and the boundary edges and faces are endowed with the metric induced from the face of $X$ they belong to.

\begin{definition}
Two metric model geometries $M_1$ and $M_2$ are  \emph{isomorphic} if they are ``germ isometric'' in the sense that there exist open sets $V_1$ and $V_2$, respectively containing the cores of $M_1$ and $M_2$, and an isometry $V_1\stackrel\simeq\to V_2$. Such an isometry is called an isomorphism (a ``germ isometry'') from $M_1$ to $M_2$. 
\end{definition}

Note that isomorphisms be composed in the obvious way (restricting domains and ranges appropriately, since the intersection of two neighbourhood of the core remains a neighbourhood of the core), which defines a category (in fact, a groupoid) of metric model geometries. The isomorphism classes of this groupoid are the ``germs'' of metric model geometry. A set of model geometries is finite if the set of germs (the quotient space of the groupoid) is finite.  This is the case for geometries of  finite type as discussed in the forthcoming section.  

We note that it is possible for a given 2-complex to be the core of two non-isomorphic model geometries.

\section{Types} \label{S - types}

We are interested in surgery constructions that preserve the ``local data'' attached to a given group or space. 
There are various ways to prescribe the local data, and we work with the following notion.

\begin{definition}\label{D - types}
A  \emph{combinatorial type}  (in dimension 2) is
\begin{enumerate}[a)]
\item a set of graphs, and 
\item a set of shapes.
\end{enumerate}
Similarly, a \emph{metric type}  (in dimension 2) is
\begin{enumerate}[a)]
\item a set of metric graphs, and 
\item a set of flat shapes. 
\end{enumerate}
\end{definition}

With this notion of type, the category $\Bord_A$ associated with a type $A$ describes surgeries that preserve both the links and the set of shapes. 

In the above definition the convention is that graphs have no orientation and that multiple edges are allowed. By ``shape'' we mean a disk with a fixed polygonal boundary, considered up to simplicial homeomorphism. 
By ``flat shape'' we mean a convex polygonal disk in the Euclidean plane $\IR^2$ endowed with the induced metric, up to isometry.

\begin{example}[combinatorial types]\mbox{}
\begin{enumerate}
\item  ``2-complexes'': \begin{enumerate}[a)]\item all graphs \item   all shapes\end{enumerate}
\item  ``simplicial 2-complexes'': \begin{enumerate}[a)]\item all graphs \item $\{\bigtriangleup\}$ (one triangle)\end{enumerate}
\item  ``simplicial surfaces'': \begin{enumerate}[a)]\item $\{n$-gons for all $n\geq 2\}$ \item $\{\bigtriangleup\}$ \end{enumerate} 
\end{enumerate}
\end{example}

\begin{example}[Metric types]\mbox{}
 \begin{enumerate}
 \item  type $\tilde A_2$ (see \cite{ronan1989lectures}): \begin{enumerate}[a)]\item incidence graphs of finite projective planes, all edges have length $\pi/3$ \item one equilateral triangle \end{enumerate} 
 \item  type $A_\MK$  (which describes the Moebius--Kantor groups, see \cite[\ts 4]{rd}): \begin{enumerate}[a)]\item the Moebius-Kantor graph, all edges have length $\pi/3$ \item one equilateral triangle  
 \end{enumerate} \end{enumerate} 
 \end{example}

This definition of type is flexible enough to include the groups of intermediate rank, and restrictive enough to make the description of the group cobordism category of a given type a non trivial task in general.
(The standard types for Bruhat--Tits buildings, in the Euclidean rank 2 case, namely, $\tilde A_2$, $\tilde B_2$ and  $\tilde G_2$, can also be defined  in the same way, as is the type  ``Tits geometries'', using a) $\{$spherical buildings$\}$ and b)  $\{\bigtriangleup\}$.)

\begin{definition}
A type is \emph{finite} if the two sets a) and b) are finite, when considered up to isomorphism (up to isometry in the metric case), and
 every graph in a) is finite.
\end{definition}

Every 2-complex $X$ is natually affiliated with a type $A_X$, called the type of $X$, and defined by 
\begin{enumerate}[a)]
\item the set of links at vertices in $X$  not in the topological boundary of $X$, and 
\item the set of shapes $X$ contains
\end{enumerate} 
This is a metric type if $X$ is metric, where the links are endowed with the angular metric. 

\begin{remark}
  Although we will not need this variation in the present paper, it would be natural to also take into account the angles in the definition of types, in order to force certain angles in the set of shapes to match with certain edges in the set of  links. 
\end{remark}

We will say that:
\begin{itemize}
\item[-] $X$ is \emph{of strict type $A$} if $A_X=A$
\item[-] $X$ is \emph{of type $A$} if $A_X\subset A$. 
%\item[-] $X$ is \emph{of embedded type $A$} if it can be embedded (homeomorphically in the simplicial case, and isometrically in the metric case) in a complex of type $A$.
\item[-] $X$ is \emph{of finite type} if $A_X$ is a finite type. 
\end{itemize}

Note that one can consider abstractly the union of two types. For example, one defines the metric type  $A_\MK+\tilde A_2$ as follows: 
\begin{enumerate}[a)]
\item $\{$incidence graph of the Fano$\}\cup\{$Moebius-Kantor graph$\}$, and 
\item one equilateral triangle. 
\end{enumerate}

\begin{definition}
Let $A$ be a type. The \emph{space of complexes of type $A$} is defined by
\[
\Omega_A:=\text{the set of isomorphism classes of }
\]
\[
\text{connected complexes of type } A\text{ without boundary}
\]
\end{definition}

 The space of complexes of finite types is:
\[
\Omega_\ft:=\text{the set of isomorphism classes of }
\]
\[
\text{connected complexes of finite type without boundary}
\]
It is filtered by the lattice (with respect to inclusion) of finite types: if $A\subset B$ are finite types then $\Omega_A\subset \Omega_B\subset \Omega_\ft$.

Note that there are countably many combinatorial types, and uncountably many metric types up to type isomorphism. Accordingly, there is a metric and a simplicial space of complexes of finite types, together with a forgetful map $\Omega_{\ft}^\metr\surj \Omega_{\ft}^\simp$ whose fibres describe the possible metrizations.

The so-called pointed Gromov--Hausdorff topology provides a natural topology on  ``desingularized'' versions $\Lambda_A$ of $\Omega_A$: 
\begin{center}
\begin{tikzpicture}
\matrix (m) [
matrix of math nodes,
row sep=.3cm,
column sep=.7cm,
% text height=1.5ex,
% text depth=0.25ex
] {
\Lambda_A \\
\Omega_A\\
};
\path[->] (m-1-1) edge node[above]{} (m-2-1);
\end{tikzpicture}
\end{center}
It is easy to find such ``desingularizations'' of $\Omega_A$  by  \emph{marking} the complexes. Thus, marking at vertices gives 
\[
\Lambda_A:=\text{the set of base point preserving isomorphism classes of pairs }(X,*)
\] 
where $X$ runs over the complexes of type  A and $*$ (the base point) runs over the vertices of $X$.
One can also choose larger balls or compact sets as ``base points'', if one wants to add local control to the isotropy groups.   

One checks readily (by a standard diagonal argument) that 

\begin{proposition}
If $A$ is a finite type (metric or simplicial), then $\Lambda_A$ is a compact Hausdorff space.
\end{proposition}

\begin{remark} 
 If $\tilde A_{2,q}$ denotes the type ``$\tilde A_2$ and order $q$'', then $\Omega_{\tilde A_{2,q}}^\text{sc}$, namely the subset of $\Omega_{\tilde A_{2,q}}$ consisting of simply connected spaces,  
coincides with the space of triangle buildings $E_q$ from \cite{E1,E2,Esurv}.      
\end{remark}

Since model geometries are in particular 2-complexes, they have a well-defined affiliated type (both combinatorial and metric). However, if a model geometry $M$ sits in an ambient 2-complex $X$, the type of $M$ differs from that of $X$ (due to the presence  of additional  faces  lying outside  the core).  We shall refer to the type of a 2-complex $X$ in which  $M$ embeds  as an \emph{embedded type} for $M$.

 A  type, if finite,  can be ``precomputed'', in the sense that one can give (at least in principle), an exhaustive list of all the model geometries it contains, thereby listing the basic building blocks for groups of a this type:
  
 \begin{proposition}
 If $A$ is a finite (combinatorial or metric) type, then the set of model geometries of embedded type $A$ is finite up to  isomorphism. 
 \end{proposition}

 \begin{proof}
 For $X$ of type $A$ and $x\in X$ a vertex, let $\overline M_x$ denote the minimal subcomplex containing all faces containing $x$. Since $A$, either combinatorial or metric,  is finite, the set of all complexes $\overline M_x$ when $X$ runs over the complexes of type $A$ and $x$ over the vertices of $X$ is finite, respectively up to simplicial isomorphism or isometry. If $M$ is a model geometry of embedded type $A$, it embeds   in $\overline M_x$ for some $x\in X$. There are only finitely many such embeddings up to  isomorphism.   
 \end{proof}
 
 In practice, the precomputation of all model geometries can only be achieved  for the smallest types (a computer program can be written that outputs the model geometries for a given type, but we do not have an efficient algorithm---see \cite[\ts 4]{rd}). 
 
 Furthermore,  even when if given type can be fully precomputed, the question remains to understand what complexes can be built from this finite list of model geometries.  The decomposition 
\[
X \simeq\left (\bigsqcup_{x\in X^0} M_x \sqcup \bigsqcup_{i=1}^n s_i  \right) /\sim
\]  
described in \ts\ref{S - model geometry}  formulates  the issue, but by no means addresses it. In a \ts\ref{S - 2/3 transitive}, we  introduce a notion of $\frac 2 3$-transitivity,  that goes one step further.  

These considerations raise several  questions. One is to find general conditions on a given type $A$ that insures the existence of sufficiently many compact complexes (in particular, at least one) of strict type $A$.  For example, if $A$ and $B$ are two given simplicial types, is there always a compact connected  2-complex of strict type $A + B$ provided that there are such complexes of  strict types $A$ and $B$, respectively? This certainly requires the types $A$ and $B$ to be ``compatible'', in a combinatorial sense. For the two types $A=\tilde A_2$ and $B=A_\MK $, which are combinatorially compatible, this question is discussed in \ts  \ref{S - 158}, and the resulting spaces are typical examples of  ``spaces of intermediate rank'', in the spirit of \cite{rd}.
Another question is the accessibility problem raised in \ts \ref{S - cobordisms}.

\section{Collars}\label{S - Collars}

\begin{definition}
An \emph{open collar} is a topological space of the form $H\times (0,1)$ where $H$ is a graph (not necessarily connected).
\end{definition}

  In the metric case, $H\times (0,1)$ is also assumed to be endowed with a metric, which may not be a product metric, but  fibers over the graph $H$. 
    The graph $H$ is called the \emph{nerve} of the collar. (From the point of view of topological quantum field theory, the assumption that the nerve is a graph corresponds to restricting our attention to ``closed strings''---it is possible but more complicated to also include ``open strings'' in the discussion.)

\begin{definition}
An open collar in a 2-complex $X$ is an embedding $C\colon H\times (0,1)\inj X$.
\end{definition} 

We shall refer to the domain $H\times (0,1)$ as the abstract collar defining  $C$. The  \emph{dual} of an open collar is the open collar $C'\colon H\times (0,1)\inj X$ defined by $C'(x,t):=C(x,1-t)$.

\begin{definition}
Let $C$ be an open collar in a 2-complex $X$. The \emph{collar closure} of $C$  the topological closure $\overline C$ of the image of $C$ in $X$.
\end{definition}

In general, collar closures are not homeomorphic to product spaces. 

\begin{definition}
The \emph{span} of a collar $C$ in  $X$ is the set $\spn(C)$  of vertices of $X$ contained in collar closure of $C$.
\end{definition}

Let $C$ be an open collar in a 2-complex $X$. The \emph{simplicial closure} of $C$ is  is the union of  all the open edges and open faces it intersects.  

We shall only consider collars in $X$ that are:
\begin{enumerate}
\item simplicially closed, in the sense that the image of the map $C$ coincide with the simplicial closure, and
\item vertex free, in the sense that they do not intersect vertex set of $X$.
\end{enumerate}
From now on, by collar in a 2-complex $X$, we refer to open collars $C\colon H\times (0,1)\to X$ that satisfies these two conditions. 

\begin{definition}
We say that two collars $C\colon H\times (0,1)\to X$ in $X$ and $C'\colon H\times (0,1)\to X$ $X'$ are isomorphic if there is  a simplicial isomorphism $\varphi\colon \overline C\to \overline C'$ such that $\varphi\circ C=C'$.     
\end{definition}

 In the metric case, the collars in $X$ are naturally endowed with the induced metric (for which the map $C$ is an isometry), and we further assume  in the previous definition  that the simplicial isomorphism  is a simplicial isometry.

\begin{example}\label{Ex - cayley} A collar $C$ in a 2-complex $X$ is \emph{conical} if it is a cone over its nerve. 
 Vertex neighbourhoods provide simple examples of conical collars. For instance, let $r$ be a face of perimeter $n$, and consider the $n$ triangle faces associated with the barycentric subdivision of $r$. Then the union of all open triangles associated with these relations together with their open boundary edges, form a  collar with corresponding abstract collar $\SI^1\times (0,1)$.
\end{example}

A  collar $C$ in $X$ with nerve $H$ and collar closure $\overline C\subset X$ can be extended to a map

\begin{center}
\begin{tikzpicture}
\matrix (m) [
matrix of math nodes,
row sep=.7cm,
column sep=.7cm,
% text height=1.5ex,
% text depth=0.25ex
] {
H\times [0,1]\\
\overline C \\
};
\path[->] (m-1-1) edge node[left,above] {}(m-2-1);
\end{tikzpicture}
\end{center}
\noindent The images of $H\times \{0\}$ (resp.\ $H\times \{1\}$) in $\overline C$ are subgraphs of the 1-skeleton of $X$  denoted $\del^- C$ and $\del^+C$ respectively. 

\begin{definition}
We call $\del^-C$ (resp. $\del^+C$) the left (resp.\ right) boundary of $C$.
\end{definition}

Observe that in general the graphs $\del^- C$ and $\del^+ C$ are not isomorphic to the nerve. In fact, they need not be homotopy equivalent to it. Edges in $H$ may disappear, or may be turned into bouquets of circles.  

\begin{definition}\label{D - boundary-injective}
We say that a collar $C$ in $X$ is \emph{boundary injective} if no edge in $\del^-C$ (resp.\ $\del^+C$) belongs to two faces of $C$. 
\end{definition}

\begin{definition}
We say that a collar $C$ in $X$ has \emph{disjoint boundary} if $\del^-C$ and $\del^+C$ are disjoint subsets of $X$.
\end{definition}

\begin{definition}
We say that a collar $C$ in $X$ is \emph{full} if it is {open} as a subset of $X$. 
\end{definition}

(Therefore, if $C$ is a full collar in $X$, then all disk adjacent to an open edge $e$ of $C$ are contained in $C$.)

 \begin{definition}\label{D - h -collar}
We say that a collar $C$ in $X$ is an \emph{h-collar} if the maps $[0,1] \ni t \mapsto C(\cdot, t)$ is a homotopy equivalence between $\del^-C$ and $\del^+C$.   
\end{definition}

 \begin{definition}
 If $X$ is a 2-complex, we write $\Col(X)$ for the set of collar of $X$ which are full with disjoint boundary. 
\end{definition}

In this paper we shall only consider collars which are nonconical.

\begin{example}
Assume that $G\acts X$ acts freely with exactly two orbits of vertices. Let $\Co M$ and $\Co N$ denote the (closed) cores of the two model geometries $M$ and $N$ in $X/G$. Then $C:=X/G\setminus (\Co M\cup \Co N)$ is a collar in $\Col(X/G)$. 
\end{example}

\begin{definition}
If $A$ is a type, we let $\Col _A$ be the set of all collars in $\Col(X)$, considered up to isomorphism, where $X$ runs over the 2-complexes of type $A$. 
\end{definition}

The set $\Col _A$ will serve as the object set of the group cobordism category $\Bord_A$.

\bigskip

Observe that in general (even in the metric case) the embedding
 \[
C\colon   H\times (0,1)\to X
 \] 
does not embed the graph $H\times \{t_0\}$, for some $t_0\in (0,1)$, as a totally geodesic subset of $X$. 

\begin{definition}
We say that a metric collar $C$ in  $\Col(X)$ is \emph{totally geodesic collar}  if $C(H, t_0)$ is a totally geodesic subset of $X$ for some $t_0\in (0,1)$. 
\end{definition}

Totally geodesic collars provide more options for metric surgery, including, for example, blow-ups of the graph $C(H,t_0)$ into a metric product $C(H,t_0)\times [0,1]$  (``collar dilatation'') that preserves non positive curvature (but possibly alters the metric type of the complex).

\begin{example}
The union of all bowties and losenges for the group $\G_\bowtie$ considered in \cite{bbelge} defines a totally geodesic collar. 
\end{example}

Associated with $A$ is a set $\Ner_A$ of topological (as opposed to metric) graphs, called the nerve space of $A$, which is the image of $\Col_A$ under the nerve map $C\mapsto H$. The latter provides a forgetful map 
\[
\Col_A\to \Ner_A.
\]
Note that  even in the metric case, the graphs in $\Ner_A$ need not be endowed with a natural metric, except in a few cases:

\begin{definition}
A type $A$ \emph{splits} if a metric can be found on every graph in $\Ner_A$ such that every elements in $\Col_A$ is a metric product, for a suitable metric on $(0,1)$.
\end{definition}

\begin{example}\label{E - square split type} 
The type $A$ defined by
\begin{enumerate}[a)]\item 4-cycles with edges of length $\pi/2$ \item   Euclidean unit square\end{enumerate}
splits. The nerve space $\Ner_A$ consists of metric graphs which are finite disjoint unions of circles of integer length.  
\end{example}

\section{$\frac 2 3$-transitivity}\label{S - 2/3 transitive}

Let $X$ be \emph{simplicial}  of dimension 2.

\begin{definition}
A group action $G\acts X$ is \emph{$\frac 2 3$-transitive} if  every (triangle) face of $X$ intersects at most two orbits of vertices.   
\end{definition}

Every vertex-transitive action is $\frac 2 3$-transitive, as is every action with two orbits of vertices. 

There are easy examples with arbitrarily large compact quotients $X/G$. 
(Let $X$ be the Euclidean plane tessellated by equilateral triangles, then the three simplicial actions $\ZI\acts X$ translating the three directions are $\frac 2 3$-transitive.) 

Furthermore, every group admits a $\frac 2 3$-transitive action on a simplicial complex of dimension 2, and, if finitely presented, a $\frac 2 3$-transitive action with compact quotient.
(Let $G$ be a group and $X$ be a Cayley complex for $G$, let $X_{\bigtriangleup}$ be a barycentric subdivision of $X$, then the action $G\acts X_{\bigtriangleup}$ is $\frac 2 3$-transitive.)

 However, for every finite type $A$ as defined in \ts \ref{S - types}, $\frac 2 3$-transitivity is related  to some global form of indicability (existence of an infinite abelian quotient  $G\surj \ZI$). 
 While not being strict indicability in the usual sense, this notion can still be profitably combined with some uniform version of the standard properties, such as uniform property T (compare \cite{Esurv}).  
 
  We have, for instance, that:
 
 \begin{theorem}\label{T - buildings 23}
The family of groups which admit a free $\frac 2 3$-transitive action on a  Bruhat-Tits building of type $\tilde A_2$ and order $q$ is finite. 
\end{theorem}

We prove this in \ts\ref{S - Global Indicability}.

\begin{remark}
A complete classification of the groups defined by Theorem \ref{T - buildings 23},  for every fixed prime power $q$, seems out of reach. It is also non trivial to find an asymptotic estimate of their number as $q\to \infty$.
\end{remark}

\begin{definition}
Two vertices $x$ and $y$ of $X$ are said to be \emph{separated} by a collar $C$ if there exists a neighbourhood $V$  containing $x,y$ and $C$ such that $V\setminus C$ has exactly two connected components, one containing $x$ and one containing $y$.
\end{definition}

For example, in a simplicialized 2-torus, the simplicial closure of an embedded circle not intersecting the vertex set defines a collar $C$ separating the vertices of $\del^- C$ from the vertices in $\del^+ C$.   

\begin{lemma}\label{L -separating full collars}
If $C$ is a full collar in $X$ with disjoint and connected boundaries, and $x,y$ are vertices in $\del^-C$ and $\del^+C$ respectively, then $C$ is separating $x$ and $y$. 
\end{lemma}

\begin{proof}
Since $C$ has compact closure and $\del^-C$ and $\del^+ C$ are disjoint, we may choose  disjoint neighbourhoods $V^-$ and $V^+$  of $\del^-C$ and $\del^+ C$, respectively. Furthermore, since $\del^-C$ and $\del^+ C$ are connected, we may assume that $V^-$ and $V^+$ are. Since $C$ is full, the set $V=V^-\cup C\cup V^+$ is a neighbourhood of $C$ containing $x$ and $y$. By construction, $V\setminus C$ has exactly two connected components containing $x$ and $y$ respectively.
\end{proof}

\begin{definition}We say that $C$ is a separating collar if it is full and if it has disjoint connected boundaries. 
\end{definition}

\begin{lemma}\label{L - separating collar}
Let $G\acts X$ be free and $\frac 2 3$-transitive. Any two distinct adjacent vertices $x,y$ of $X/G$ are separated by a collar, namely, the union of all open edges and open triangles whose vertex set closure coincides with $\{x, y\}$.
\end{lemma}

\begin{proof}
By definition, if  $G\acts X$ is $\frac 2 3$-transitive then every triangle in $X/G$ contains at most two vertices, and therefore the following property holds: every inscribed polygon (in the sense of \ts\ref{S - model geometry}) in a crossing face  of $X/G$ is degenerate and reduced to a segment.  The union of these segments over all the  triangles containing both $x$ and $y$ defines a graph $H$ in $X/G$ (which may not be connected). The resulting collar $C$ is the union of these triangles. Each triangle meets two points $x$ and $y$ of $X$, and the collar is separating $x$ from $y$ in the sense of the previous definition, by Lemma \ref{L -separating full collars}, since it is open, with disjoint connected boundaries.
\end{proof}

Let $G\acts X$ be a $\frac 2 3$-transitive action. We define a graph $Z$ as follows:
\begin{itemize}
\item Vertex set of $Z:=$ vertex set of $X/G$
\item Edge set of $Z:=$ pairs $(x,y)$ of adjacent vertices in $X/G$.
\end{itemize}
and consider the map
\[
\pi\colon X/G\to Z
\]
by sending every simplex of the model geometry over $x$ to $x$, and every simplex in the collar over $(x,y)$ to the open edge $(x,y)$.

\begin{definition}
The map $\pi\colon X/G\to Z$ (resp.\  the graph $Z$) is called the \emph{stack}  (resp.\ \emph{base}) of a $\frac 2 3$-transitive action $G\acts X$.  
\end{definition}

The map $\pi$ is a stack  
  in the sense of the Bass--Serre theory \cite{serre1977arbres}, as discussed in \cite[Chap.\ 6]{geoghegan2007topological}, where the nerves provide the edge fibers of the stack.

The stack $\pi\colon  X/G\to Z$ fails to provide a graph of groups decomposition for the group $G$ in general, and accordingly, the groups from   Bass--Serre theory bear little resemblance to groups of intermediate rank, in general, due to the missing $\pi_1$-injectivity assumption. 
Algebraically, the stack map $\pi$  provides little information beyond that which is already contained in the Seifert--van Kampen theorem. 

\begin{definition}
We say that a collar $C$ in $X$ separating two vertices $x$ and $y$ has \emph{$\pi_1$-injective boundary} if the maps $\del^-C\to M_x$ and $\del^+C\to M_y$ into the model geometry at $x$ and $y$ respectively are $\pi_1$-injective.  
\end{definition}

\section{Indicability criterion}\label{S - Global Indicability}

In this section we describe a ``global'' indicability criterion, which concerns the existence of an infinite abelian quotient  $G\surj \ZI$  that takes place in  $\Omega_A$, where the group $G$ may vary and is obtained from  surgery.

\begin{definition}\label{D _ G type A} We say that a group $G$ is \emph{of type $A$} if it admits a free action $G\acts X$ with compact quotient on  a complex $X$ of type $A$.
\end{definition}

(It would also be appropriate to include proper actions in this definition.)

\begin{theorem}\label{T - 23 transitive}
  Let $A$ be a finite combinatorial type. Assume that there exist infinitely many pairwise nonconjugate free $\frac 2 3$-transitive actions  with compact quotients on simplicial complexes of type $A$.  Then there exists an indicable group of type $A$.
\end{theorem}

\begin{proof}
Let $G\acts X$ be free and $\frac 2 3$-transitive on a simplicial complex of type $A$ with stack
\[
\pi\colon X/G\to Z.
\]
For every pair $(x,y)$ of distinct adjacent vertices in $X/G$, let $C_{x,y}$ be the separating collar between $x$ and $y$ (Lemma \ref{L - separating collar}), and, for every vertex $x$ in $X/G$, let $M_x$ be the model geometry at $x$ in $X/G$. (The collars $C_{x,y}$ are not assumed to have $\pi_1$-injective boundaries.) 

Let us write 
\[
M_x\stackrel{C_{x,y}}\lra M_y
\]
for the corresponding configuration in quotient space $X/G$ (which corresponds to a closed edge in $Z$ under $\pi$).
 By the finiteness of  $A$,  the total number of all such collar configurations  is finite up to isomorphism.

Assume  that the set of conjugacy classes of free $\frac 2 3$-transitive actions $G\acts X$, where $X$ is a simplicial complex of type $A$ with $X/G$ compact, is infinite. Notice that the stacks of conjugate actions are equivariant, and the bases are isomorphic. Furthermore, since $A$ is finite, the bases $Z$ are uniformly locally bounded, and therefore we can find arbitrary long non-backtracking segments in the bases $Z$. 

By the box principle, for one of these actions, say $G_0\acts X_0$,  one can find three disjoint consecutive edges $e, f, g$ in a segment included in $Z_0$, whose collar configurations in $X_0/G_0$ are isomorphic to a given configuration. Two of the three edges, say $e$ and $f$, are pointing in the same direction, and we can find an isomorphism between the configuration  
\[
M_x\stackrel{C_{x,y}}\lra M_y.
\ \ \ \text{and} \ \ \ M_{x'}\stackrel{C_{x',y'}}\lra M_{y'},
\]
corresponding to $e$ and $f$ respectively,
 that takes $x$ to $x'$ and $y$ to $y'$.

Consider, then,  infinitely many copies $G_p\acts X_p$, indexed by $p\in \ZI$, of this action $G_0\acts X_0$. The complex $G\acts X$ with $G\surj \ZI$ will be obtained by doing a simple surgery on the quotient spaces $X_p/G_p$ which respects the type $A$.

The surgery starts by ``duplicating the collars'' in order to be able to glue them together.
Write $X_p/G_p$ (for every $p\in \ZI$) as a quotient of a space $\hat X_p$
\[
\pi_p\colon \hat X_p\surj X_p/G_p
\]
where $\pi_p$ is the identity map outside $C_{x_p,y_p}$ and $C_{x_p',y_p'}$, and a two sheeted covered over $C_{x_p,y_p}$ and $C_{x_p',y_p'}$, resulting in a space with four copies the same collar (up to isomorphism) in the neighbourhood of its boundary, say
\[
C_{x_p,y_p}^-, C_{x_p,y_p}^+, C_{x_p',y_p'}^-, C_{x_p',y_p'}^+.
\]

We distinguish three cases, according to the number of connected components of $Z_0\setminus \{e,f\}$.  

Note that $C_{x_p,y_p}^+$ and $C_{x_p',y_p'}^-$ are in the same component.

\begin{enumerate}
\item If $Z_0\setminus \{e,f\}$ is connected, then we identify (for every $p\in \ZI$) 
\begin{enumerate}
\item[]   $C_{x_p,y_p}^+$  with $C_{x_{p+1}',y_{p+1}'}^-$, and

\item[] $C_{x_p',y_p'}^+$  with  $C_{x_{p},y_{p}}^-$.
\end{enumerate}

\item If $Z_0\setminus \{e,f\}$ has two connected components, then $\hat X_p$ has now two connected components. If the component containing  $C_{x_p,y_p}^+$ and $C_{x_p',y_p'}^-$ contains only these two collars, then we identify (for every $p\in \ZI$) 
\begin{enumerate}
\item []  $C_{x_p,y_p}^+$  with $C_{x_{p+1}',y_{p+1}'}^-$
\end{enumerate}
and discard the other component.  Otherwise, one of the two components contains 3 collars, and then we repeat the steps described in the case where $Z_0\setminus \{e,f\}$ is connected.

\item Finally, if $Z_0\setminus \{e,f\}$ has three connected components, then  the component containing  $C_{x_p,y_p}^+$ and $C_{x_p',y_p'}^-$ contains only these two collars, and we discard two connected components and repeat the steps in the case where $Z_0\setminus \{e,f\}$ has 2 connected components.  
\end{enumerate}

In all three cases, the resulting space $Y$ can be represented symbolically as an infinite chain:
\[
\ldots \lra M_{x_p} - M_{y_p} \lra M_{x'_p} - M_{y'_p} \lra M_{x_{p+1}} - M_{y_{p+1}}\lra \cdots  
\]
$\lra$ indicates the surgery operation.

Note that the only new geometric configurations in $Y$ are of the form
\[
M_x\lra M_{y'} \ \ \ \text{and} \ \ \ M_{x'}\lra M_{y}
\] 
which fit into chains of the form 
\[
M_x\stackrel{C_{x,y}}\lra M_y\simeq M_{x'}\stackrel{C_{x',y'}}\lra M_{y'}
\] 
 where the type is preserved by definition of the identification. 
Therefore, $Y$ is of type $A$.

Furthermore, the space $Y$ admits a free action of $\ZI$ by translations, whose quotient $Y/\ZI$ is compact and isomorphic to a ``double'' of $X_0/G_0$ (which needs not be a double cover). 
Namely, the quotient space $Y/\ZI$  can be represented symbolically as follows:
\begin{center}
\begin{tikzpicture}
\matrix (m) [
matrix of math nodes,
row sep=.3cm,
column sep=.7cm,
% text height=1.5ex,
% text depth=0.25ex
] {
 & M_{x'} & \\
Y/\ZI\ \  =\ \ \  M_y&&M_{y'}\\
&M_x&\\
};
\path[-] (m-1-2) edge node[above]{} (m-2-3);
\path[-] (m-2-1) edge node[above]{} (m-3-2);
\path[<->] (m-2-1) edge node[above]{} (m-1-2);
\path[<->] (m-2-3) edge node[above]{} (m-3-2);
\end{tikzpicture}
\end{center}

Let 
\[
G:=\pi_1(Y/\ZI)\text{ and } X:= \widetilde{Y/\ZI}.
\]
Since the type is clearly preserved by taking universal covers, the complex $X$ is of type $A$.
Furthermore, by construction, $G\acts X$ is a free $\frac 2 3$-transitive action with compact quotient $X/G=Y/\ZI$ admits an infinite abelian cover. In particular $G\surj \ZI$. 
(Note that the action of the group of Galois transformations of $X\surj Y$ 
 itself is   free and $\frac 2 3$-transitive.)
 \end{proof}

\begin{proof}[Proof of Theorem \ref{T - buildings 23}]
Since the type ``$\tilde A_{2,q}:=\tilde A_2$ of order $q$'' is finite, if there are infinitely many $\frac 2 3$-transitive free actions, then we can find an indicable group $G$ of type $\tilde A_{2}$. This contradicts property T---In fact, this contradicts either  Garland's theorem directly, or indeed the Cartwright--Mlotkowski--Steger theorem that such a group has Kazhdan's property T. We recall that Garland's theorem is the statement that the first cohomology group $H^1(G,\pi)$ with coefficient in a finite dimensional unitary representation $\pi$ vanishes. In particular, $H^1(G,\CI)=0$ (trivial coefficients) so that $H^1(G,\ZI)=\Hom(G,\ZI)$ is torsion. Therefore, the number of groups, and for every such group, the number of its $\frac 2 3$-transitive free action, is finite.
\end{proof}

We note that these constructions  extend beyond simplicial complexes, using the following version of $\frac 2 3$-transitivity:  

\begin{definition}
A free action $G\acts X$ on a 2-complex is \emph{mildly  transitive} if every inscribed polygon in a crossing face  of $G/X$ is reduced to a segment.
\end{definition}

 A simplicial separating collar between any two distinct adjacent vertices can defined in the same way for mildly transitive actions, and for simplicial complexes, mild  transitivity is equivalent to $\frac 2 3$-transitivity.

  The following straightforward generalization of Theorem \ref{T - 23 transitive} holds.

\begin{theorem} If  $A$ is finite and there are infinitely many free mildly transitive actions $G\acts X$, with $X$ of type $A$ and $X/G$ compact, then there exists an indicable group of type $A$.
\end{theorem}

\section{The ST lemma}\label{S - Minimal linking graph}\label{S - ST Lemma}

\begin{definition}
  A collar is \emph{thick} if every vertex of the nerve is adjacent to at least three edges. 
\end{definition}

\begin{definition} A metric complex $X$ is $\theta$-convex if the angle of every face is $<\theta$. 
 We let $A(2\pi,\theta)$ denote the metric type which describes the simplicial metric $\theta$-convex 2-complexes of  nonpositive curvature. 
\end{definition}
  
  (The type $A(2\pi,\theta)$ is uncountable.)

\begin{lemma}[ST lemma]\label{L - minimal linking graph} 
 The nerve of a minimal thick collar of $\Col_{A(2\pi,2\pi/3)}$ is isomorphic to one of the following two graphs: 
 \begin{enumerate}[a)]
 \item  the cylinder --- or ``thickened square''
 \[
 S= \begin{tikzpicture}[baseline=2.7ex]
 \coordinate (A) at (0,0);
 \coordinate (B) at (1,0);
 \coordinate (C) at (1,1);
 \coordinate (D) at (0,1);
    \draw (A) -- (B);
    \draw (C) -- (D);
    \draw (A) to [bend left] (D);
    \draw (B) to [bend left] (C);
    \draw (A) to [bend right] (D);
    \draw (B) to [bend right] (C);
  \end{tikzpicture}
 \]
or,
 \item the tetrahedron: 
 \[
 T= \begin{tikzpicture}[baseline=2ex]
 \coordinate (A) at (0,0);
 \coordinate (B) at (1,0);
 \coordinate (C) at (.5,.866);
 \coordinate (D) at (.5,.2886);
    \draw (A) -- (B) -- (C) -- cycle;
  \draw (A) -- (D); \draw (B) -- (D); \draw (C) -- (D);
  \end{tikzpicture}
 \]
 \end{enumerate} 
(We view $S$ as a square with two opposite double edges.) 
\end{lemma}

The proof is given below.
Let $X$ be a nonpositively curved $\frac {2\pi}3$-convex metric simplicial 2-complex.

Let $C$ be a thick collar in $\Col(X)$.  Since $C$ is thick, the path 
\[
\gamma_v\colon t\mapsto C(v,t)
\]
in $X$ is contained in the 1-skeleton of $X$ for every vertex $v$ of $H$.
Since $C$ is vertex-free, $\gamma_v$ does not intersect the vertex set of $X$. 

It follows that every edge $e$ in $H$ between two vertices $u,v$ is attached to a unique vertex in $\spn(C)$, defined as the intersection of the two edges $e_u$ and $e_v$ containing $\gamma_u$ and $\gamma_v$. 

This gives a surjective map
\[
H^1\surj \spn(C).
\] 
\begin{definition}
This map is called the \emph{span decomposition} of the nerve $H$. 
\end{definition}

The span decomposition consists of $|\spn(C)|$ subgraphs of $H$, denoted $H_x$ for $x\in \spn(C)$, partitioning the edge set. Since every edge in $H_x$ corresponds to a triangle attached to $x$ we have 

\begin{lemma}
For every $x\in \spn(C)$, the graph $H_x$ is isomorphic to subgraph of the link of $x$.
\end{lemma}

Pictorially, we have $|\spn(C)|$ subgraphs in the links of $x\in \spn(C)$ that ``move towards the centerpiece'' of the collar to recombine into the nerve $H$ of $C$.  

We note that:

\begin{lemma}
If $C$ is a collar and $\spn(C)=\{x\}$, then $H$ is isomorphic to a union of connected components of the link of $x$.
\end{lemma}

\begin{proof}
The nerve $H$ is isomorphic to the subgraph $H_x$ of the link $L_x$. The collar $C$ being full, if $H_x$ contains a vertex of $L_x$, it contains all the edges attached to $x$. This implies that $H_x$ is a union of connected components of $L_x$. 
\end{proof}

Since we have assume that collars in $\Col(X)$  be non conical, this shows that $|\spn(C)|\neq 1$ for every $C\in \Col(X)$.

\begin{lemma} Under the assumptions of the ST lemma, the nerve $H$ contains at least 4 vertices and 6 edges.
\end{lemma}

\begin{proof}
Let  $H$ denote the nerve of $C$. If $v$ denotes the number of vertices, then since $C$ is thick,  $H$ has $e\geq \frac 3 2 v$ edges. The three minimal cases are:

\begin{enumerate}[a)]
\item $v=1$ and $e\geq 2$
\item $v=2$ and $e\geq 3$
\item $v=3$ and $e\geq 5$.
\end{enumerate}
which we will now show are not possible.

Note that if $|\spn(C)|\geq 3$, then $v\geq 4$, since the union of three edges corresponding to  different span vertex cannot form a loop in $H$. Therefore in the three above cases $|\spn(C)|= 2$.

The graph $H$ does not contain a loop, for if it does then so does one of the graphs $H_x$ in the span decomposition, and since $H_x\subset L_x$, this contradicts nonpositive curvature as link edges have length $<\pi$ by strict convexity. This takes care of case a), and shows that every edge in the other cases has distinct extremities.

In the second case b),  at least one of the two graphs $H_x$ or $H_y$, $x,y\in \spn(C)$, must also contain a double edge, which in turn forces the link of $x$, or that of $y$, to have a double edge. The same strict convexity argument then applies.

In the last case, the graph $H$ cannot contain a triple edge, for otherwise since $|\spn(C)|= 2$ one of the two graphs $H_x$ or $H_y$, $x, y\in \spn(C)$ would contain a double edge.  Therefore we have either $v=3$ and $e=5$, or $v=3$, and $e=6$, corresponding to a triangle with, respectively, two or three double edges, and at least one of the two span graphs, say $H_x$, is a triangle (with simple edges).  This contradicts  the link condition by $\frac {2\pi}3$-convexity.

 Therefore,  $v\geq 4$ and $e\geq 6$.
\end{proof}

 \begin{proof}[Proof of the ST lemma]
It follows from the previous lemma that if $H$ is as in the ST lemma, and $H$ is a minimal solution, then $H$ is a connected cubic graph with 4 vertices and 6 edges. There are precisely two such graphs, $T$ and $S$. We will see in the forthcoming sections that these graphs indeed occur as nerves of collars.  
\end{proof}

\begin{remark}\label{R -systolic}
It is clear from the proof that the ST lemma remains true when nonpositive curvature and $\frac{2\pi}3$-convexity are replaced by a rather weak \emph{systolic condition}, namely, that $X$ is a  simplicial 2-complex whose links have girth $\geq 4$. (Note that the corresponding type $A(4)$ is countably infinite.) For the purpose of rank interpolation, which is primarily concerned with questions on nonpositively curved spaces and CAT(0) groups, the metric version is directly useful (and in the situations we have in mind,  the $\frac{2\pi}3$-convexity assumption is always satisfied).  There are also more general versions of the lemma, with a modified classification, when the systolic assumption is further relaxed.
\end{remark}

We will also need the following result.

\begin{definition}
A collar in $X$ is of type $S$ (resp.\ $T$, resp.\ $S$$T$) if its nerve is isomorphic to $S$ (resp.\ $T$, resp.\ $S$$T$).
\end{definition}

\begin{lemma}[Span decomposition of minimal collars.]\label{L-minimal collar spans} If $C$ is a  collar in $X$ of type $S$$T$ with span 
\[
\spn(C)=\{x,y\}
\] 
and if $X$ is $\frac \pi 2$-convex, then in the span decomposition the two graphs 
\[
H = H_x\cup H_y
\] 
are isomorphic to a path of length 3.  

Furthermore, the span decompositions are given by:
 \begin{enumerate}[a)]
 \item  
 \[
 S= \begin{tikzpicture}[baseline=2.7ex]
 \coordinate (A) at (0,0);
 \coordinate (B) at (1,0);
 \coordinate (C) at (1,1);
 \coordinate (D) at (0,1);
    \draw[dotted] (A) -- (B);
    \draw (C) -- (D);
    \draw[dotted] (A) to [bend left] (D);
    \draw[dotted] (B) to [bend left] (C);
    \draw (A) to [bend right] (D);
    \draw (B) to [bend right] (C);
  \end{tikzpicture}
 \]
 \item 
 \[
 T= \begin{tikzpicture}[baseline=2ex]
 \coordinate (A) at (0,0);
 \coordinate (B) at (1,0);
 \coordinate (C) at (.5,.866);
 \coordinate (D) at (.5,.2886);
    \draw[dotted] (A) -- (C);
    \draw (A) -- (B) -- (C);
  \draw (A) -- (D); \draw[dotted]  (B) -- (D); \draw[dotted]  (C) -- (D);
  \end{tikzpicture}
 \]
 \end{enumerate} 
 where (say) the dotted subgraph corresponds to $H_x$ and its complement to $H_y$.
\end{lemma} 

\begin{proof}
Since $H$ contains either a double edge or a cycle of length 3, an argument similar to that in the previous lemma shows, using $\frac{2\pi}3$-convexity, that $|\spn(C)|\neq 1$, i.e.\ $x\neq y$.

Let $H_x$ and $H_y$ be the corresponding graphs, and consider the case of $T$ first. By  $\frac{2\pi}3$-convexity, neither $H_x$ nor $H_y$ can contain a cycle of $T$, and they must therefore be (a priori possibly disconnected) trees. However, it is easy to check that in that case both $H_x$ and $H_y$ must have exactly 3 edges, and must be connected. The drawing shows the only possible embeddings up to graph isomorphism.  

Consider then the case of collars of type $S$. By strict convexity, neither $H_x$ nor $H_y$ can contain a double edge, so both of them have at least two edges. We distinguish two cases. In the first case, $H_x$ has two edges and $H_y$ four. It follows that $H_y$ corresponds to a circle of length 4 in $H$, which contradicts $\frac \pi 2$-convexity. Therefore both $H_x$ and $H_y$ must have 3 edges. Again, the drawing shows the only possible embeddings up to graph isomorphism.  
\end{proof}

\begin{definition}
We say that a collar is  \emph{treeable} if the span decomposition of its edge set consists of maximal subtrees.  
\end{definition}

Lemma \ref{L-minimal collar spans} shows that collars of type $S$$T$ are treeable.

\begin{proposition}Let $X$ be a nonpositively curved $\frac {\pi}2$-convex metric simplicial 2-complex and $C$ be a boundary injective  treeable collar in $X$ spanning 
  two vertices. Then $C$ is a $h$-collar.
\end{proposition}

This provides an easy criterion for checking if a collar of type $S$$T$ is a $h$-collar in a given complex $X$ satisfying the assumptions.

\begin{proof}
Let us write $\spn(C)=\{x,y\}$ and let
 $H:= H_x\cup H_y$ be the span decomposition of the nerve $H$.
By assumption, both $H_x$ and $H_y$ are subtrees of $H$.

Since $C$ is treeable, $H\times [0,1]$ is a homotopy between $H/H_x$ and $H/H_y$, where $H/H_x$ and $H/H_y$ denotes respectively the retract of $H$ along $H_x$ and $H_y$.

The extension of $C$ to $H\times [0,1]$, namely the map
\begin{center}
\begin{tikzpicture}
\matrix (m) [
matrix of math nodes,
row sep=.7cm,
column sep=.7cm,
% text height=1.5ex,
% text depth=0.25ex
] {
H\times [0,1]\\
\overline C \\
};
\path[->] (m-1-1) edge node[left,above] {}(m-2-1);
\end{tikzpicture}
\end{center}
\noindent whose image is the collar closure $\overline C$ of $C$ in $X$, induces two maps
$H/H_x \surj \del^-C$ 
and $H/H_y \surj \del^+C$. Since $X$ is simplicial, these maps send edges to edges, and do so injectively, by boundary injectivity. Furthermore, since $H_x$ and $H_y$ are maximal subtrees, they are graph isomorphisms, and both $\del^-C$ and $\del^+C$ are bouquets of circles.  Thus  $H\times [0,1]\surj \overline C$ is a homotopy between $\del^-C$ and $\del^+C$.
\end{proof}

Finally, we observe:

\begin{lemma}\label{L - order4}
Under the assumptions of Lemma \ref{L-minimal collar spans},
\begin{enumerate}
\item there is a graph involution $\theta\colon S\to S$ that respects the span decomposition, namely, such that 
\[
 H_x\stackrel{\theta}\lra H_y.
\]
\item there is no graph involution $T\to T$ that respects the span decomposition. 
\item there exists an element  $\sigma\colon T\to T$ of order 4 that respects the span decomposition, namely, such that 
\[
 H_x\stackrel{\sigma}\lra H_y.
\] 
\end{enumerate}
\end{lemma}

\begin{proof}
(1) is clear. For (2), note that the  symmetry group being the symmetric group on the 4 vertices,  no involution respects the span decomposition. However, if the external vertices in $T$ are labelled 1,2,3, oriented counterclock wise, with 1 on top and 4 in the center, then the 4-cycle $\sigma=(1243)$ provides the desired transformation of  order 4. 
\end{proof}

\begin{remark}
If the involution $\theta\colon S\to S$ described in (1) extends to the collar $C$ of type $S$ (for example if $X$ is made of equilateral triangles), then $C$ is self-dual. 
\end{remark}

\section{Double covers}\label{S-double cover}
 
In this section we show that  collars of type $T$ do not occur among double covers in complexes of type $A_\MK$ or $\tilde A_{2,2}$.
Double covers are a convenient way to produce filling cobordisms of the given type. 

%:
Using a simple idea explained in \ts\ref{S-fake cover} below,  this observation can be used to construct ``fake double covers'', by substituting collars of type $S$ by collars of type $T$ in plain double covers. 

Collars of type $T$ can be ruled out by a simple symmetry argument using the involution acting on double covers. 

In fact, the argument leads naturally to new collars in such complexes, including collars of the following two types:

\begin{definition} Let us call $\Theta$-nerve  (resp.\ $\Theta'$-nerve) the graph
\[\Theta:=\ 
 \begin{tikzpicture}[baseline=-.5ex]
 \coordinate (A) at (0.3,0);
 \coordinate (B) at (1,1);
 \coordinate (C) at (1,0);
 \coordinate (D) at (1,-1);
 \coordinate (E) at (2,1);
 \coordinate (F) at (2,0);
 \coordinate (G) at (2,-1);
 \coordinate (H) at (2.7,0);
    \draw (A) -- (B);
    \draw (A) -- (C);
    \draw (A) -- (D);
    \draw (B) to [bend left] (E);
    \draw (C) to [bend left] (F);
    \draw (D) to [bend left] (G);
    \draw[dotted] (E) -- (H);
    \draw[dotted] (F) -- (H);
    \draw[dotted] (G) -- (H);
    \draw[dotted] (B) to [bend right] (E);
    \draw[dotted] (C) to [bend right] (F);
    \draw[dotted] (D) to [bend right] (G);
  \end{tikzpicture}
  \ \ \ \ (\text{resp. } \Theta^\times:=\ 
 \begin{tikzpicture}[baseline=-.5ex]
 \coordinate (A) at (0.3,0);
 \coordinate (B) at (1,1);
 \coordinate (C) at (1,0);
 \coordinate (D) at (1,-1);
 \coordinate (E) at (2,1);
 \coordinate (F) at (2,0);
 \coordinate (G) at (2,-1);
 \coordinate (H) at (2.7,0);
    \draw (A) -- (B);
    \draw (A) -- (C);
    \draw (A) -- (D);
    \draw (B) -- (F);
    \draw (C) -- (E);
    \draw (D) to [bend left] (G);
    \draw[dotted] (E) -- (H);
    \draw[dotted] (F) -- (H);
    \draw[dotted] (G) -- (H);
    \draw[dotted] (B) to  (E);
    \draw[dotted] (C) to  (F);
    \draw[dotted] (D) to [bend right] (G);
  \end{tikzpicture})
 \]
with the indicated span decomposition. A collar is  of type $\Theta$ (resp. $\Theta'$) if its nerve is the $\Theta$-nerve (resp. $\Theta'$-nerve). 
\end{definition}

These are larger collars which behave similarly  to collars of type $S$$T$ in this context.  We will see that collars of type $\Theta'$ also do not appear among double covers in complexes of type $A_\MK$ or $\tilde A_{2,2}$.

\bigskip

Let us fix some notation for this section:

\begin{itemize}
\item $X$ denotes a metric 2-complex which is either of type $A_\MK$ or $\tilde A_{2,2}$ with one vertex.
\item $X'$ denotes an arbitrary 2-cover of $X$
\item $C$ denotes the collar separating the two vertices $x$ and $y$ in $X'$, and $H$ refers to the nerve of 
$C$ with its nerve decomposition. 
\end{itemize}

\begin{lemma}\label{L - nerve 6 12}
The nerve $H$ has either 6 or 12 vertices. Furthermore,  $\Aut(H)$ contains an involution $s$ exchanging the two components of the span decomposition. 
\end{lemma}

\begin{proof}
The nerve $H$ has $v$ vertices and $e= \frac 3 2 v$ edges.
Furthermore, if $e_x$ and $e_y$ denotes the number of edges of $H_{x}$ and $H_{y}$ respectively, so $e=e_x+e_y$, then we have
\[
2e_x+e_y\leq |L_y|\leq 24
\]
\[
2e_y+e_x\leq |L_x|\leq 24
\]
by assumption, therefore $e\leq 16$. Since $e$ is a multiple of 3, it follows that $e=6, 9,12,15$. 

By assumption, $ \pi _1(X)$ admits an index 2 subgroup acting $\frac 2 3$-transitively on $ X $ with quotient $ X'$, which gives the desired  involution 
\[
s \colon H\to H.
\]
It is clear that this involution takes $H_{x}$ to $H_{y}$. In particular, $e_x=e_y$ so $e$ is even,  narrowing down the  options to $e=6$ or $e=12$.  
\end{proof}

This proves our first claim:

\begin{proposition}\label{P - not type T}
The collar $C$ is not of type $T$.   
\end{proposition}

\begin{proof}
Assume that that nerve $H$ contains 6 edges. By Lemma \ref{L - minimal linking graph}, $H$ as a nerve coincides with $S$ or with $T$ with the given span decomposition. The second case is ruled out by the fact that the involution $s$ must permute the two copies of the path of length 3 in the span decomposition. (Compare Lemma \ref{L - order4}.)
\end{proof}

We are lead to conside collars with 12 edges. In this case, we  assume furthermore that $H$ is connected. 
A priori, disconnected collars may appear, and the connected components, being collars themselves, must have at least 6 edges by Lemma 
\ref{L - minimal linking graph}. In particular both connected components must be of type ST. We will come back to this interesting situation later.

We need to explain now how collars of type $\Theta$ and $\Theta'$ arise. Then the collars of type $\Theta'$ can be ruled out using symmetry. 

\begin{definition}
We say that a collar is cubic if its nerve is isomorphic to the cube.
\end{definition}

The following proves our second claim, that collars of type $\Theta'$ do not appear. 

\begin{lemma}\label{L - H 12 edges}
Assume that $H$ has 12 edges, and that the components of the nerve decomposition are connected.  Then $C$ is either cubic or of type $\Theta$. In particular, $C$ is not of type $\Theta'$.
\end{lemma}

\begin{proof}
Using the notation of Lemma \ref{L - nerve 6 12}, if $e=12$, then $e_x=e_y=6$, and using the symmetry $s$ we see that $H_{x}$ and $H_{y}$ are either
\begin{itemize}
\item two circles of length 6, or 
\item two isomorphic trees with 6 edges.
\end{itemize}
Furthermore, by our assumption, the two trees in the second case are connected. 

The first case is in fact not possible, for one of the vertices of $H$ would have order 4. This same argument also works  in the second case to show that both trees can't be straight segments. 

Note that by definition the involution $s$ cannot fix an edge of $H$. Neither can it fix a vertex, since $H$ is cubic. Let $a,b,c$ be the number of vertices of order 1,2,3 respectively in $H_x$ and $H_y$. Note that $a=b$ since since every $x$-vertex of order 1 can be paired with a $y$-vertex of order 2. Furthermore,
\[
a+2b+3c=12
\]
(counting edges in $H_x$ or $H_y$) so
\[
a+c=4.
\]
If $H_x$ is connected then
\[
a+b+c=7
\]
so $a=b=3$ and $c=1$. Therefore $H_x$ and $H_y$ are tripods of height 2. It is not hard to check that there are three ways to combine these tripods together to give a cubic graph on 8 vertices. The three possibilities are the cube,  and the two  graphs $\Theta$ and $\Theta'$. However, the graph $\Theta'$ does not admit an involution that permutes $H_x$ and $H_y$.
\end{proof}

We now consider the case where the components of the span decomposition are disconnected. 

\begin{lemma}\label{L - octagonal collar}
If the components of the span decomposition are disconnected, then they have two connected components, both of which being segments of length 3. 
\end{lemma}

\begin{proof}
In the notation of the previous lemma, if $H_x$ has $k\geq 2$ connected components connected then
\[
a+b+c=6+k
\]
so $a=b=2+k$ and $c=2-k\geq 0$. This forces $k= 2$, $a=b=4$ and $c=0$. In that case we have either that
\[
H_x\simeq H_y\simeq P_3\sqcup P_3
\]
or that
\[
H_x\simeq H_y\simeq P_5\sqcup P_1
\]
where $P_n$ denotes the segment with $n$ edges. An argument similar to the case of a single segment $\simeq P_6$ in Lemma \ref{L - H 12 edges} shows that this is not possible in a connected graph $H$. 
\end{proof}

\begin{definition}
A collar is said to be \emph{octagonal} if its nerve is isomorphic to the following graph 
\[
\begin{tikzpicture}[shift ={(0.0,0.0)},rotate = 67.5]
\tikzstyle{every node}=[font=\small]

\coordinate (v1) at (1.0,0.0);
\coordinate (v2) at (0.71,0.71);
\coordinate (v3) at (-0.0,1.0);
\coordinate (v4) at (-0.71,0.71);
\coordinate (v5) at (-1.0,-0.0);
\coordinate (v6) at (-0.71,-0.71);
\coordinate (v7) at (0.0,-1.0);
\coordinate (v8) at (0.71,-0.71);
\draw[solid,thin,color=black,-] (v1) -- (v2);
\draw[solid,thin,color=black,-] (v2) to [bend left] (v3);
\draw[dotted,thin,color=black,-] (v2) to [bend right] (v3);
\draw[dotted,thin,color=black,-] (v3) -- (v4);
\draw[solid,thin,color=black,-] (v4) to [bend left] (v5);
\draw[dotted,thin,color=black,-] (v4) to [bend right] (v5);
\draw[solid,thin,color=black,-] (v5) -- (v6);
\draw[solid,thin,color=black,-] (v6) to [bend left] (v7);
\draw[dotted,thin,color=black,-] (v6) to [bend right] (v7);
\draw[dotted,thin,color=black,-] (v7) -- (v8);
\draw[solid,thin,color=black,-] (v8) to [bend left] (v1);
\draw[dotted,thin,color=black,-] (v8) to [bend right] (v1);
\end{tikzpicture}
\]
with the given span decomposition.
\end{definition}

Lemma \ref{L - octagonal collar} implies:

\begin{lemma}\label{L - H 12 edges}
Assume that $H$ has 12 edges, and that the components of the nerve decomposition are disconnected.  Then $C$ is octagonal. 
\end{lemma}

To summarize, we have shown that:

\begin{proposition}\label{P - 2cover final} 
In a 2-cover (with connected nerve) of a 2-complex of type $A_\MK$ or $\tilde A_{2,2}$ with one vertex, the collar is either:
\begin{itemize}
\item  of type $S$
\item cubic
\item of type $\Theta$
\item octagonal
\end{itemize}
In particular, the collars of types $T$ and $\Theta'$ do not appear.
\end{proposition}

\section{Group cobordisms}\label{S - cobordisms}

\begin{definition}\label{D - group cobordisms}
A \emph{group cobordism} is a 2-complex $X$ together with a pair $(C,D)$ of collars in $\Col(X)$ whose boundaries $\del^-C$ and $\del^+ D$ form a partition of the topological boundary of $X$:
\[
\del X= \del^-C\sqcup \del ^+ D.
\] 
\end{definition} 

We call $C$ and $D$ respectively the domain and range of $X$, and use the notation 
$X\colon C\to D$ to denote cobordisms of domain $C$ and range $D$. We view $C$ and $D$ as subcomplexes of $X$. Namely, we assume that $X$ is endowed with fixed injective 2-complex morphisms $L_X\colon \overline C\inj X$ and $R_X\colon \overline D\inj X$. The \emph{collar boundary} of $X$ is defined by 
\[
\del_c X:=\overline C +\overline D.
\]
Here and below, $+$ indicates set-theoretic union. If either $\overline C$ or $\overline D$ is empty, we say that $X$ is a null-cobordism or a filling.
If $X=C=D$ we call $X$ the identity cobordism from $C$ to $D$. In this paper we assume that every group cobordism $X\colon C\to D$ which is not the identity  is such that $\del^+C$ and $\del^-D$ are included in the interior of $X$.   

We shall call \emph{dual} of a group cobordism $X\colon C\to D$ is the group cobordism $X'\colon D'\to C'$ where $X$ coincide with $X'$ as a topological space, and $C',D'$ are the respective dual of $C$ and $D$. We have a partition
\[
\del X' = \del^- D' \sqcup \del ^+ C'.
\]

\begin{definition}\label{D - cobordism isomorphism}
A morphism (resp.\ isomorphism) of group cobordisms is 2-complex morphism (resp.\ isomorphism) that takes collar boundaries to collar boundaries. 
 \end{definition}

Let $A$ be a type (combinatorial or metric). 
We define a small category $\Bord_A$ as follows: 
\begin{itemize}
\item Object set: the set $\Col_A$ of all isomorphism classes of collars in 2-complexes $X$ of  type $A$, which we view as 2-complexes $C,D,E,\ldots$
\item Arrow set:  the set of all equivalence classes of group cobordisms of type $A$, where $X\colon C\to D$ and $Y\colon C\to D$ are said to be equivalent if there is a 2-complex isomorphism $\p \colon X\to Y$, such that $\p\circ L_X = L_Y$ and $\p\circ R_X = R_Y$.
\end{itemize}

The composition in $\Bord_A$ is defined by amalgamation as in a standard cobordism category. If $X\colon C\to D$ and $Y\colon D\to E$ are group cobordisms, then 
\[
Y\circ X\colon C\to E
\]
is a group cobordism, with domain $C$ and range $E$, given by
\[
Y\circ X = (X \sqcup Y)/\sim_{\overline D}
\]
where  $\sim_{\overline D}$ identifies the two copies of $\overline D$ using $L_Y\circ R_X^{-1}$, and 
\[
\del (Y\circ X) = \del^-C^-\sqcup \del^+E^+.
\]
This gives a well-defined arrow from $C$ to $E$, and composition is strict. The following lemma holds because our notion of type is local.

\begin{lemma}
If $X$ and $Y$ are group cobordisms of  type $A$ then $Y\circ X$ is a group cobordism of  type $A$.
\end{lemma}

\begin{proof}
 Since both $X$ and $Y$ are of embedded type $A$, and $\sim_D$ identifies the two copies of $D$, the set of shapes in $Y\circ X$ is indeed described by $A$. To check the link condition, we only need to consider  the vertices in the boundary of $\overline D$ (viewed as sitting in $Y\circ X$) which do not belong to the boundary of $Y\circ X$. If  $x\in \del^-D$ (resp.\ $y\in \del D^+$) is such a vertex, then it is an interior vertex of $X$ (resp.\ of $Y$), and therefore its link is indeed described by $A$. 
\end{proof}

Observe that the collar  closure of $C\colon H\times (0,1)\to X$ defines itself a group cobordism of  type $A$, which corresponds to the unit arrow in the category $\Bord_A$.

\begin{definition}
We call $\Bord_A$ the (unoriented) group cobordism category of type $A$.
\end{definition}

As in the classical case, the category $\Bord_A$ is a symmetric monoidal category, even if it is not, strictly speaking, a  cobordism category in the classical sense \cite{stong1968notes}. 
Nevertheless, the abstract cobordism relation (see \cite[Chapter I]{stong1968notes}) defines an equivalence relation on $\Bord_A$. 

\begin{definition}
Two objects $C,D$ in $\Bord_A$ are cobordant, in symbols, $C\equiv D$, if there exists two arrows $X, Y$ in $\Bord_A$  such that 
\[
\overline C + \del_c X \simeq \overline  D+ \del_c Y
\]
where $\simeq$ denotes collar isomorphism.
\end{definition}

Note that if $C\equiv D$ and $E\equiv F$ then $C+D\equiv E+F$.

\begin{lemma}
The relation $\equiv$ is an equivalence relation.
\end{lemma}

\begin{proof}
$X=Y=\emptyset$ provides reflexivity, symmetry is obvious, and if $C\equiv  D$ and $D\equiv E$, then choosing arrows $X,Y,Z,T$ such that
\[
\overline C + \del_cX \simeq \overline D+ \del_cY\text{  and } \overline D + \del_cZ \simeq \overline E+ \del_cT
\]
then using the arrows $X+Z$ and $Y+T$ we see that
\[
\overline C+\del_c X + \del_cZ\simeq \overline D+ \del_c Y+\del_c Z\simeq \overline E+ \del_c T+\del_c Y
\]
so $C\equiv E$.
\end{proof}

\begin{definition}
Let $A$ be a type (topological or metric). The (unoriented) group cobordism monoid of type $A$ is the set
\[
\Omega_A:=\Bord_A^0/\equiv
\] 
where $\Bord_A^0$ denotes the object set and the monoid structure is given by
\[
[C]+[D]:=[C + D]
\] 
with neutral element $[\emptyset]$.
\end{definition}

It is obvious that:

\begin{lemma}
$\Omega_A$ is an abelian monoid.
\end{lemma}

The following elementary proposition, combined with Prop. \ref{P - 2cover final}, allows to show that some classes vanish in $\Omega_\MK $ (and $\Omega_{\tilde A_2}$).

\begin{proposition}\label{P - vanishing collar in Omega} Let $A$ be a type. If $C$ be the separating collar in a complex of type $A$ with two vertices, then 
\[
[C]=[C']=0
\] 
in the group cobordism monoid $\Omega_{A}$. (Here $C'$ denotes the dual collar.) 
\end{proposition}

\begin{proof}
By assumption we can write $X = (X^-\sqcup X^+)/\sim_{\overline C}$, where $\del_c X^- = \overline C$ and $\del_c X^+ = \overline C'$. 
\end{proof}

\begin{definition}
We say that a group of type $A$ is \emph{accessible by surgery} if it is isomorphic to the fundamental group of quotient of the form $X/\sim$, where $X\colon C\to C$ is a cobordism of type $A$, $C$ is an object in $\Bord_A$, and $\sim$ identifies the two copies of $C$ in $X$. 
\end{definition}

Accessibility raises  natural questions on $A$; are there, for instance, conditions on  $A$ which ensures that `most' (resp., `almost no') groups of  type $A$ are accessible by surgery  (`most' may refer to `all' or `all but finitely many', and `almost no' to `none', or `finitely many', for example).

\section{Fake double covers}\label{S-fake cover}
 
In this section we construct ``fake'' double covers for Moebius--Kantor complexes\ using a simple surgery in the following way: 
 \begin{enumerate}
 \item assume that $X'\surj X$ is a double cover of a complex $X$ of rank \sq\ with 1 vertex, whose separating collar $C$ is (say, connected) of type $S$
 \item choose a collar $D$ of type $T$ whose boundary is isomorphic to $\del C$:
 \[
 \del^-D=\del^-C\text{ and }\del^+D=\del^+D,
 \] 
 \item substitute $D$ to $C$ in $X'$. 
 \end{enumerate} 

 This  leads to a new complex $X''$, which is ``of fake rank \sq''.  
 Here is an explicit construction.

 The classification in \cite[\ts 4]{rd} provides several complexes $X$ of rank \sq\ with one vertex such that $H_1(X,\ZI/2\ZI)\neq 0$. One of the simplest examples is the complex (denoted $V_2^3$ in \cite[\ts 4]{rd}) defined by 
\[
X := [[1,1,3],[2,2,4],[1,5,2],[3,6,4],[3,7,6],[4,6,8],[5,7,8],[5,8,7]]
\]   
whose rational homology is reduced to $\ZI$. 
The corresponding group of rank \sq\ has a presentation of the following form: 
\[
\G := \langle s,t\mid           s^2t^3s^3t^2=   t^2  s^2,\
    t^2=s^2t^2s^2t^2s^{-2} ts
 \rangle.
\]
The group $\G$ admits an index 2 subgroup $\G'$, which appears as the fundamental group of a covering space $X'\surj X$ of degree 2 and can be explicitly computed to be:
\[
X' := [[1,11,3],[2,12,4],[1,15,12],[3,6,4],[3,7,6],[4,6,8],[5,7,8],[5,8,7],
\]
\[
[11,1,13],[12,2,14],[11,5,2],[13,16,14],[13,17,16],[14,16,18],[15,17,18],[15,18,17]]
\]  
where the new edges are labeled $1x$ for every edge with label $x$ in $X$. 

A direct computation shows:

\begin{lemma}
The separating collar $C$ in $X'$ is of type $S$.
\end{lemma}

In order to construct a ``fake'' double cover, we have to substitute to $C$ a collar $D$ of type $T$.

The collar closure of $C$ in $X'$ is the 2-complex given by:
\[
[[1,11,3],[2,12,4],[1,15,12],[11,1,13],[12,2,14],[11,5,2]].
\]
We will consider instead the following collar
\[
[[1,-2,3],[-11,12,4],[1,15,12],[11,1,13],[12,2,14],[11,5,2]].
\]
where $-x$ indicate that the edge $x$ is opposite in the triangle. Note that, by definition, this construction flips the appropriate edges in the nerve of the collar $C$, which will therefore become a collar of type $T$. 

In other words, we define a new complex $X''$ as follows:
\[
X'' := [[1,-2,3],[-11,12,4],[1,15,12],[3,6,4],[3,7,6],[4,6,8],[5,7,8],[5,8,7],
\]
\[
[11,1,13],[12,2,14],[11,5,2],[13,16,14],[13,17,16],[14,16,18],[15,17,18],[15,18,17]]
\]  

By construction we have:

\begin{lemma}
The separating collar $D$ in $X''$ is of type $T$.
\end{lemma}

The flip in the nerve of $X'$ corresponds to a transformation of its links which exchanges the extremities of two edges. This transformation can easily be computed explicitly:

\begin{lemma}
The links in $X''$ are isomorphic to:
\begin{figure}[h]
\centering
\begin{tikzpicture}[shift ={(0.0,0.0)},scale = 2.0]
\tikzstyle{every node}=[font=\small]

\coordinate (v1) at (1.0,0.0);
\node [right] at (v1) {$\ \ \ \simeq\ \ \ $};
\coordinate (v2) at (0.92,0.38);
\coordinate (v3) at (0.71,0.71);
\coordinate (v4) at (0.38,0.92);
\coordinate (v5) at (-0.0,1.0);
\coordinate (v6) at (-0.38,0.92);
\coordinate (v7) at (-0.71,0.71);
\coordinate (v8) at (-0.92,0.38);
\coordinate (v9) at (-1.0,-0.0);
\coordinate (v10) at (-0.92,-0.38);
\coordinate (v11) at (-0.71,-0.71);
\coordinate (v12) at (-0.38,-0.92);
\coordinate (v13) at (0.0,-1.0);
\coordinate (v14) at (0.38,-0.92);
\coordinate (v15) at (0.71,-0.71);
\coordinate (v16) at (0.92,-0.38);
\draw[solid,thin,color=black,-] (v1) -- (v2);
\draw[solid,thin,color=black,-] (v2) -- (v3);
\draw[solid,thin,color=black,-] (v3) -- (v4);
\draw[solid,thin,color=black,-] (v4) to[bend left] (v6);
\draw[solid,thin,color=black,-] (v5) -- (v6);
\draw[solid,thin,color=black,-] (v5) to[bend left] (v7);
\draw[solid,thin,color=black,-] (v7) -- (v8);
\draw[solid,thin,color=black,-] (v8) -- (v9);
\draw[solid,thin,color=black,-] (v9) -- (v10);
\draw[solid,thin,color=black,-] (v10) -- (v11);
\draw[solid,thin,color=black,-] (v11) -- (v12);
\draw[solid,thin,color=black,-] (v12) -- (v13);
\draw[solid,thin,color=black,-] (v13) -- (v14);
\draw[solid,thin,color=black,-] (v14) -- (v15);
\draw[solid,thin,color=black,-] (v15) -- (v16);
\draw[solid,thin,color=black,-] (v16) -- (v1);
\draw[solid,thin,color=black,-] (v1) -- (v6);
\draw[solid,thin,color=black,-] (v3) -- (v8);
\draw[solid,thin,color=black,-] (v5) -- (v10);
\draw[solid,thin,color=black,-] (v7) -- (v12);
\draw[solid,thin,color=black,-] (v9) -- (v14);
\draw[solid,thin,color=black,-] (v11) -- (v16);
\draw[solid,thin,color=black,-] (v13) -- (v2);
\draw[solid,thin,color=black,-] (v15) -- (v4);
\shade[ball color=black] (v1) circle (.02);
\shade[ball color=black] (v2) circle (.02);
\shade[ball color=black] (v3) circle (.02);
\shade[ball color=black] (v4) circle (.02);
\shade[ball color=black] (v5) circle (.02);
\shade[ball color=black] (v6) circle (.02);
\shade[ball color=black] (v7) circle (.02);
\shade[ball color=black] (v8) circle (.02);
\shade[ball color=black] (v9) circle (.02);
\shade[ball color=black] (v10) circle (.02);
\shade[ball color=black] (v11) circle (.02);
\shade[ball color=black] (v12) circle (.02);
\shade[ball color=black] (v13) circle (.02);
\shade[ball color=black] (v14) circle (.02);
\shade[ball color=black] (v15) circle (.02);
\shade[ball color=black] (v16) circle (.02);
\end{tikzpicture}
\begin{tikzpicture}[shift ={(0.0,0.0)},scale = 2.0]
\tikzstyle{every node}=[font=\small]

\coordinate (v1) at (1.0,0.0);
\coordinate (v2) at (0.92,0.38);
\coordinate (v3) at (0.71,0.71);
\coordinate (v4) at (0.38,0.92);
\coordinate (v6) at (-0.0,1.0);
\coordinate (v5) at (-0.38,0.92);
\coordinate (v7) at (-0.71,0.71);
\coordinate (v8) at (-0.92,0.38);
\coordinate (v9) at (-1.0,-0.0);
\coordinate (v10) at (-0.92,-0.38);
\coordinate (v11) at (-0.71,-0.71);
\coordinate (v12) at (-0.38,-0.92);
\coordinate (v13) at (0.0,-1.0);
\coordinate (v14) at (0.38,-0.92);
\coordinate (v15) at (0.71,-0.71);
\coordinate (v16) at (0.92,-0.38);
\draw[solid,thin,color=black,-] (v1) -- (v2);
\draw[solid,thin,color=black,-] (v2) -- (v3);
\draw[solid,thin,color=black,-] (v3) -- (v4);
\draw[solid,thin,color=black,-] (v4) -- (v6);
\draw[solid,thin,color=black,-] (v5) -- (v6);
\draw[solid,thin,color=black,-] (v5) -- (v7);
\draw[solid,thin,color=black,-] (v7) -- (v8);
\draw[solid,thin,color=black,-] (v8) -- (v9);
\draw[solid,thin,color=black,-] (v9) -- (v10);
\draw[solid,thin,color=black,-] (v10) -- (v11);
\draw[solid,thin,color=black,-] (v11) -- (v12);
\draw[solid,thin,color=black,-] (v12) -- (v13);
\draw[solid,thin,color=black,-] (v13) -- (v14);
\draw[solid,thin,color=black,-] (v14) -- (v15);
\draw[solid,thin,color=black,-] (v15) -- (v16);
\draw[solid,thin,color=black,-] (v16) -- (v1);
\draw[solid,thin,color=black,-] (v1) -- (v6);
\draw[solid,thin,color=black,-] (v3) -- (v8);
\draw[solid,thin,color=black,-] (v5) -- (v10);
\draw[solid,thin,color=black,-] (v7) -- (v12);
\draw[solid,thin,color=black,-] (v9) -- (v14);
\draw[solid,thin,color=black,-] (v11) -- (v16);
\draw[solid,thin,color=black,-] (v13) -- (v2);
\draw[solid,thin,color=black,-] (v15) -- (v4);
\shade[ball color=black] (v1) circle (.02);
\shade[ball color=black] (v2) circle (.02);
\shade[ball color=black] (v3) circle (.02);
\shade[ball color=black] (v4) circle (.02);
\shade[ball color=black] (v5) circle (.02);
\shade[ball color=black] (v6) circle (.02);
\shade[ball color=black] (v7) circle (.02);
\shade[ball color=black] (v8) circle (.02);
\shade[ball color=black] (v9) circle (.02);
\shade[ball color=black] (v10) circle (.02);
\shade[ball color=black] (v11) circle (.02);
\shade[ball color=black] (v12) circle (.02);
\shade[ball color=black] (v13) circle (.02);
\shade[ball color=black] (v14) circle (.02);
\shade[ball color=black] (v15) circle (.02);
\shade[ball color=black] (v16) circle (.02);
\end{tikzpicture}

\caption{The fake Moebius--Kantor graph}
\end{figure}
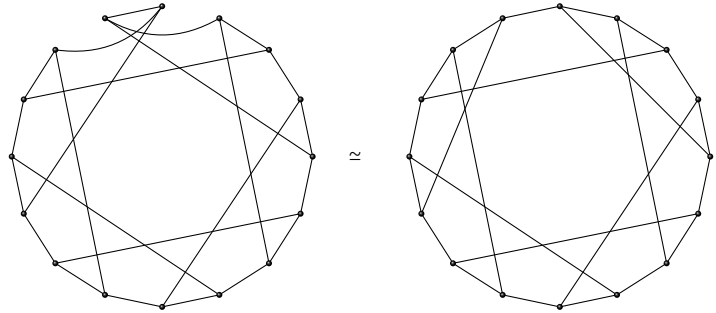
\end{lemma}

Note that the fake Moebius--Kantor graph has girth 5, and therefore the space $X''$ admits  sections of positive curvature; 
at the same time, part of the flatness is turned into negative curvature. A simplicial 2-complex whose link are isomorphic to the fake Moebius--Kantor graph is said to be of fake rank \sq.  

The lack of symmetries of collars of type $T$ (compare the proof of Prop.\ \ref{P - not type T}) shows that: 

\begin{proposition}\label{P - X'' not double cover}
The complex $X''$ is not a double cover. 
 \end{proposition} 
 
In general, a 2-complex with two vertices  fails to be a double cover for various reasons, which can be detected in balls of sufficiently large radius around the vertices;  the most obvious reason is to have non isomorphic links at the vertices (see  \ts\ref{S - 158} for the existence of such complexes).

\section{Non filling group cobordisms of Moebius--Kantor type}

Model geometries, as defined in  \ts \ref{S - model geometry}, can be classified, and give rise to the simplest group cobordisms. In the present section we classify the nonfilling such cobordisms, for groups of Moebius--Kantor type. In turn, these cobordisms can be used to construct new groups of intermediate rank.
  
  The main result is that:

\begin{theorem} \label{T - cobordism 74}
There exist precisely two non filling group cobordisms of rank \sq\ with one vertex whose boundary collars are connected of type $ST$. Furthermore, the two cobordisms are self-dual, and their collars are pairwise isomorphic,  self-dual, and of type $S$.
\end{theorem}

Observe that such cobordisms do not exist, for example, for groups (of rank 2) of type $\tilde A_2$.

\begin{proof}[Proof of Theorem \ref{T - cobordism 74}] 
Let $X\colon C\to D$ be such a group cobordism, and $L$ be the link of the unique vertex in $X$ not in the boundary. By assumption, $L$ is isomorphic to the Moebius--Kantor graph (link of rank \sq).

By definition of the collar of type $S$ (see Lemma \ref{L - minimal linking graph}), we have two disjoint embeddings of the path of length 3 in $L$, given by the span decomposition. Let $\alpha$ and $\beta$ denote the two copies of these two paths. 

Observe that $\alpha$ and $\beta$ are roots in  $L$ in the sense of \cite{chambers}. Since $L$ is the graph of rank \sq, these roots have rational rank
 \[
  \rk(\alpha),\rk(\beta)\in \left\{\frac 3 2, 2\right\}.
\]

\begin{lemma}\label{L - aut two orbits roots} The automorphism group of $L$  has exactly two orbits of roots, which are the set of roots of rank 2 and the set of roots of rank $\frac 3 2$. 
\end{lemma}

\begin{proof}
The argument can be extracted from the proof of Proposition 41 in \cite{rd}. Namely, the automorphism group $G$ is transitive on the flags $A\subset \gamma$ where $\gamma$ is a simplicial path of length 2 and $A$ is an  extremity of $\gamma$, from which it follows that there are at most 2 orbits of roots. It is clear on the other hand that roots of rank 2 and roots of rank $\frac 3 2$ belong to distinct $G$-orbits.  
\end{proof}

\begin{lemma} 
The open 1-neighbourhoods of $\alpha$ and $\beta$ are disjoint: $N_1(\alpha)\cap N_1(\beta)=\emptyset$.
\end{lemma}
\begin{proof} If not, an  edge connecting $\alpha$ and $\beta$ would correspond to a triangle whose vertices are distinct embedded vertices of $X$, i.e., form three distinct vertices in a complex of rank \sq\ in which $X$ embeds. Therefore such an edge corresponds to an edge in  the boundary of $X$ that connects the collars $C$ and $D$, which contradicts the fact that $\del X= \del^-C\sqcup \del ^+ D$.     
\end{proof}

\begin{lemma}\label{L - loops descriptions} The loops in $X$  (viewed as a model geometry, as defined in \ts \ref{S - model geometry}) can be of three types:
\begin{itemize}
\item disjoint from  $N_1(\alpha \cup \beta)$
\item connecting two vertices in $N_1(\alpha)$
\item connecting two vertices in $N_1(\beta)$
\end{itemize}  
\end{lemma}

\begin{proof}
This follows from the fact that if a loop intersects $N_1(\alpha)$, then it is the boundary to an embedded triangle of $X$ having an angle in $N_1(\alpha)$. Such a triangle  belongs to $C$, and its boundary loop connects two vertices in $N_1(\alpha)$.  
\end{proof}

By Lemma \ref{L - aut two orbits roots} it is enough to consider the two cases
\begin{enumerate}[a)]
\item $\rk(\alpha)= \frac 3 2$, and,
\item $\rk(\alpha)= 2$.
\end{enumerate}  
We start with b). The following lemma is straightforward.

\begin{lemma}\label{L - rank 2 case}
If $\rk(\alpha)=2$, then  $L \setminus N_2(\alpha)$ is a segment of simplicial length 5.
\end{lemma}

  This gives us 3 possibilities for the position of $\beta$. One corresponds to a root of rank 2, and the two other cases, which are permuted by a graph automorphism, to a root of rank $\frac 3 2$. 
  
  In fact, the position of $\beta$ in $L$ is uniquely determined relative to $\alpha$:
  
\begin{lemma}\label{L - rank 2 case beta symmetric}
If $\rk(\alpha)=2$, then $\beta$ is the unique root of rank 2 included in $L \setminus N_2(\alpha)$.
\end{lemma}
  
\begin{proof}  
Assume that $\beta$ is one of the two roots of rank $\frac 3 2$. Let  $u$ denote the unique vertex of $L$ which is at distance 2 from both $\alpha$ and $\beta$. The unique loop $\gamma$ of $M$ through  $x$ intersect $L$ at a vertex $v$, which must be at distance 1 from either $\alpha$ or $\beta$, which is a contradiction. 
\end{proof}

Let us now consider a). We have the following analog of Lemma \ref{L - rank 2 case}:

\begin{lemma}
If $\rk(\alpha)= \frac 3 2$, then  $L \setminus N_2(\alpha)$ is the disjoint union of a root and an edge.
\end{lemma} 

In particular, the root $\beta$ is uniquely determined relative to $\alpha$.
Note that in both cases a)  and b) we have:

\begin{lemma}  $\rk(\alpha)=\rk(\beta)$.
\end{lemma}

We shall refer to a) as the rank $\frac 3 2$ case and to b) as the rank $2$ case. They are represented   on the left hand side and respectively  the right hand side in the figure below.

\begin{center}
\includegraphics[width=4cm]{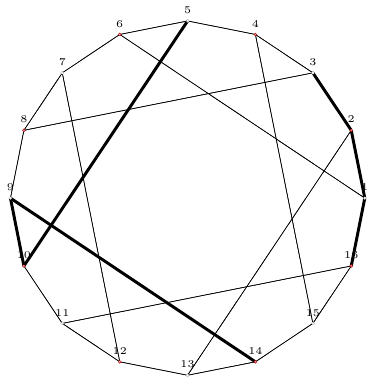}\hskip1cm 
\includegraphics[width=4cm]{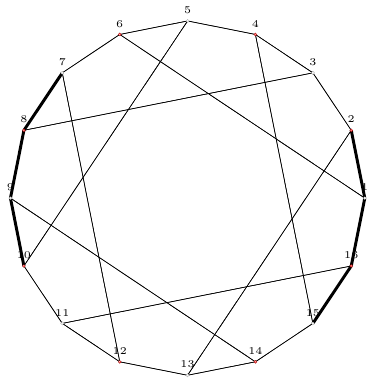} 
\end{center}
 
In order to prove Theorem \ref{T - cobordism 74},  we have to show that there exists precisely one group cobordism for each of these two cases. 
We consider the rank $\frac 3 2$ case  (left) first.

By Lemma \ref{L - loops descriptions}, the loops connect points in the 1-neighbourhood of the roots, and there exists a unique edge at distance 2 from both roots (edge $(7,12)$ in the figure). This edge must have its two extremities in a loop, say $\gamma_1$. 

We have to show that:

\begin{lemma}\label{L - gamma determined} There exists a unique choice for the three remaining loops $\gamma_2,\gamma_3,\gamma_4$ in $X$, up to permutation of the symbols.
\end{lemma}
 
\begin{proof}
A triangle in the core of $X$ can either be glued on 2 or on 3 distinct loops.   
Since $\gamma_1$ connects two adjacent edges of $L$, it must contains a core triangle of each type. We denote $\gamma_2$ the loop connecting the edges of the first triangle face (say $f$) and $\gamma_3,\gamma_4$ the loops connecting the edges of the second triangle face (say $g$). 

There are two solutions for $\gamma_2$, namely $(6,11)$ or $(8,13)$, which are symmetric. 
\begin{center}
\includegraphics[width=4cm]{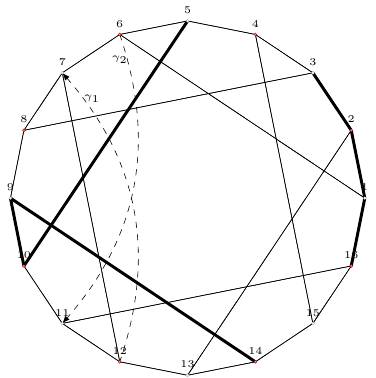}\hskip1cm
\includegraphics[width=4cm]{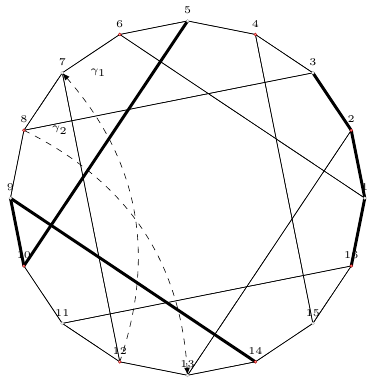}
\end{center}
The link support of $g$ is therefore, respectively, 
\[
[(12,13),(7,8),(4,15)]\text{ or, } [(11,12),(6,7),(4,15)].
\]
We have to show that these two configurations exclude collars of type $T$, but are compatible with collars of type $S$. We will show that there is a unique way to extend the first configuration, and by symmetry, and unique way to extend the second, which are therefore isomorphic as group cobordisms. The unique extension is represented in Figure \ref{F - Cobordism 32}.

A direct computation shows  that the collars of type $S$ and $T$ induce the following local geometry in $X$: 

\begin{figure}[h]
\begin{tikzpicture}[line join=bevel,z=-5.5]
\tikzstyle{every node}=[font=\tiny]

\coordinate (a) at (0,3);
\coordinate (b) at (0,2);
\coordinate (c) at (0,1);
\coordinate (d) at (0,0);

\coordinate (e) at (1,3);
\coordinate (f) at (1,2);
\coordinate (g) at (1,1);
\coordinate (h) at (1,0);

\coordinate (e1) at (2,3.5);
\coordinate (e2) at (2,3);
\coordinate (f1) at (2,2);
\coordinate (g1) at (2,1);
\coordinate (h1) at (2,0);
\coordinate (h2) at (2,-0.5);

\draw[solid,line width=0.3mm,color=black,-] (a) -- (b);
\draw[solid,line width=0.3mm,color=black,-] (b) -- (c);
\draw[solid,line width=0.3mm,color=black,-] (c) -- (d);
\draw[solid,line width=0.3mm,color=black,-] (e) -- (f);
\draw[solid,line width=0.3mm,color=black,-] (f) -- (g);
\draw[solid,line width=0.3mm,color=black,-] (g) -- (h);
\draw[solid,line width=0.1mm,color=black,-] (e) -- (e2);
\draw[solid,line width=0.1mm,color=black,-] (e) -- (e1);
\draw[solid,line width=0.1mm,color=black,-] (f) -- (f1);
\draw[solid,line width=0.1mm,color=black,-] (g) -- (g1);
\draw[solid,line width=0.1mm,color=black,-] (h) -- (h1);
\draw[solid,line width=0.1mm,color=black,-] (h) -- (h2);

\draw[dashed,line width=0.1mm,color=black,-] (a) -- (f);
\draw[dashed,line width=0.1mm,color=black,-] (b) -- (e);
\draw[dashed,line width=0.1mm,color=black,-] (c) -- (h);
\draw[dashed,line width=0.1mm,color=black,-] (d) -- (g);
\draw[dashed,line width=0.1mm,color=black,-] (e2) to[bend left] (f1);
\draw[dashed,line width=0.1mm,color=black,-] (g1) to[bend left] (h1);
\draw[dashed,line width=0.1mm,color=black,-] (e1) to[bend left] (h2);

\draw (a) node[below left =10pt]{$a$};
\draw (b) node[below left =10pt]{$b$};
\draw (c) node[below left =10pt]{$c$};

\draw (e) node[above right =10pt]{$b$};
\draw (e) node[right =10pt]{$a$};
\draw (f) node[right =10pt]{$a$};
\draw (g) node[right =10pt]{$c$};
\draw (h) node[right =10pt]{$c$};
\draw (h) node[below right =10pt]{$b$};

\shade[ball color=black] (a) circle (0.2ex);
\shade[ball color=black] (b) circle (0.2ex);
\shade[ball color=black] (c) circle (0.2ex);
\shade[ball color=black] (d) circle (0.2ex);
\shade[ball color=black] (e) circle (0.2ex);
\shade[ball color=black] (f) circle (0.2ex);
\shade[ball color=black] (g) circle (0.2ex);
\shade[ball color=black] (h) circle (0.2ex);
\shade[ball color=black] (e1) circle (0.2ex);
\shade[ball color=black] (e2) circle (0.2ex);
\shade[ball color=black] (f1) circle (0.2ex);
\shade[ball color=black] (g1) circle (0.2ex);
\shade[ball color=black] (h1) circle (0.2ex);
\shade[ball color=black] (h2) circle (0.2ex);

\end{tikzpicture}
\hskip1cm
\begin{tikzpicture}[line join=bevel,z=-5.5]
\tikzstyle{every node}=[font=\tiny]

\coordinate (a) at (0,3);
\coordinate (b) at (0,2);
\coordinate (c) at (0,1);
\coordinate (d) at (0,0);

\coordinate (e) at (1,3);
\coordinate (f) at (1,2);
\coordinate (g) at (1,1);
\coordinate (h) at (1,0);

\coordinate (e1) at (2,3.5);
\coordinate (e2) at (2,3);
\coordinate (f1) at (2,2);
\coordinate (g1) at (2,1);
\coordinate (h1) at (2,0);
\coordinate (h2) at (2,-0.5);

\draw[solid,line width=0.3mm,color=black,-] (a) -- (b);
\draw[solid,line width=0.3mm,color=black,-] (b) -- (c);
\draw[solid,line width=0.3mm,color=black,-] (c) -- (d);
\draw[solid,line width=0.3mm,color=black,-] (e) -- (f);
\draw[solid,line width=0.3mm,color=black,-] (f) -- (g);
\draw[solid,line width=0.3mm,color=black,-] (g) -- (h);
\draw[solid,line width=0.1mm,color=black,-] (e) -- (e2);
\draw[solid,line width=0.1mm,color=black,-] (e) -- (e1);
\draw[solid,line width=0.1mm,color=black,-] (f) -- (f1);
\draw[solid,line width=0.1mm,color=black,-] (g) -- (g1);
\draw[solid,line width=0.1mm,color=black,-] (h) -- (h1);
\draw[solid,line width=0.1mm,color=black,-] (h) -- (h2);

\draw[dashed,line width=0.1mm,color=black,-] (a) -- (f);
\draw[dashed,line width=0.1mm,color=black,-] (b) -- (h);
\draw[dashed,line width=0.1mm,color=black,-] (c) -- (e);
\draw[dashed,line width=0.1mm,color=black,-] (d) -- (g);
\draw[dashed,line width=0.1mm,color=black,-] (e1) to[bend left] (h2);
\draw[dashed,line width=0.1mm,color=black,-] (e2) to[bend left] (g1);
\draw[dashed,line width=0.1mm,color=black,-] (f1) to[bend left] (h1);

\draw (a) node[below left =10pt]{$a$};
\draw (b) node[below left =10pt]{$b$};
\draw (c) node[below left =10pt]{$c$};

\draw (e) node[above right =10pt]{$b$};
\draw (e) node[right =10pt]{$c$};
\draw (f) node[right =10pt]{$a$};
\draw (g) node[right =10pt]{$c$};
\draw (h) node[right =10pt]{$a$};
\draw (h) node[below right =10pt]{$b$};

\shade[ball color=black] (a) circle (0.2ex);
\shade[ball color=black] (b) circle (0.2ex);
\shade[ball color=black] (c) circle (0.2ex);
\shade[ball color=black] (d) circle (0.2ex);
\shade[ball color=black] (e) circle (0.2ex);
\shade[ball color=black] (f) circle (0.2ex);
\shade[ball color=black] (g) circle (0.2ex);
\shade[ball color=black] (h) circle (0.2ex);
\shade[ball color=black] (e1) circle (0.2ex);
\shade[ball color=black] (e2) circle (0.2ex);
\shade[ball color=black] (f1) circle (0.2ex);
\shade[ball color=black] (g1) circle (0.2ex);
\shade[ball color=black] (h1) circle (0.2ex);
\shade[ball color=black] (h2) circle (0.2ex);

\end{tikzpicture}
\caption{Local geometry induced by a collar of type $S$ (left)  and  a collar of type $T$ (right)}\label{F - local geometry collar}
\end{figure}
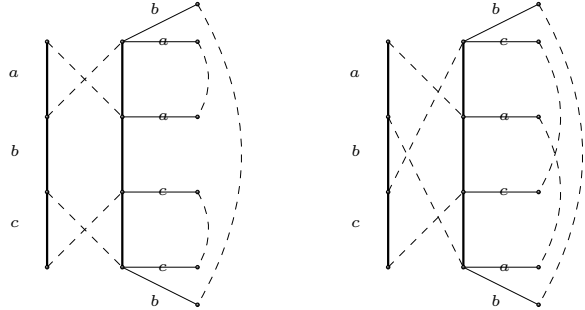

In view of the local geometry, for both types of collars, the loop $\gamma_3$ is determined by the given configuration to have support $(4,13)$, which in turn, forces $\gamma_4$ (up to permutation of indices) to have support $(8,15)$. 
\end{proof}

Lemma \ref{L - gamma determined} determines the faces $a$, $b$, $c$ of the left collar. They are shown in Figure \ref{F - Cobordism 32}. Note that $a$ and $c$ are adjacent in the root $\alpha$, which proves that the left hand side collar  is of type $S$. Furthermore, it turns out that it is indeed possible to extend this configuration by adding three more faces $a'$, $b'$, $c'$ belonging to the right hand side collar. Again, $a'$ and $c'$ are adjacent, and the right hand side collar must be of type $S$.

We summarize these results as follows.

\begin{figure}[h]
\includegraphics[width=5cm]{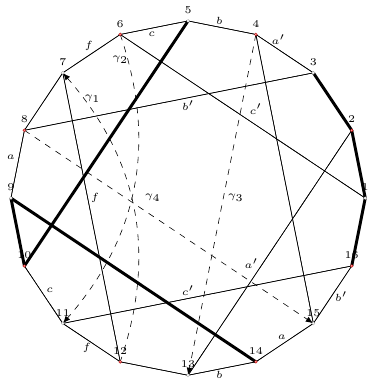}
\caption{The first (rank $\frac 3 2$) group cobordism of Moebius--Kantor type (drawing of the core).}\label{F - Cobordism 32}
\end{figure}

It remains to show that:

\begin{lemma} The first group cobordism of Moebius--Kantor type is  self-dual.
\end{lemma}

\begin{proof}
Assuming that the duality automorphism exist, it has to exchange the roots and in particular it can only send 3 to 5 or to 14. Since it must also preserve the loops, so in particular it must fix $\{7,12\}$ as a set.  Since $d(3,7)=2$ and $d(3,12)=3$, we have either $3\to 5, 7\to 7$ or $3\to 14, 7\to 12$.   The former case implies $16\to 14,12\to12$, so 15 is fixed, it follows that 8 is also fixed, and so is the tripod at 7, which is a contradiction. Therefore, the partial permutation
\[
(3,2,1,16,7)\lra (14,9,10,5,12)
\] 
is determined, and particular, $8\lra 13$ and $6\lra 11$, so $\gamma_1$ and $\gamma_2$ are preserved (reversing the orientation), while $\gamma_3$ and $\gamma_4$ are permuted. It follows that the automorphism can be written as the following product of eight transpositions:
\[
(3,14)(2,9)(1,10)(16,5)(7,12)(8,13)(6,11)(4,15).
\] 
Conversely, one checks that this indeed defines an automorphism of the given group cobordism (in the sense of Def.\ \ref{D - cobordism isomorphism}).
\end{proof}

We shall now turn to the case $\rk(\alpha)=2$.

By Lemma \ref{L - rank 2 case}, $L \setminus N_2(\alpha)$ is a segment of simplicial length 5 and, by Lemma \ref{L - rank 2 case beta symmetric}, the same is true of $L \setminus N_2(\beta)$. Let us call $\overline \beta$ and $\overline \alpha$ these two segments, respectively. The extremities of $\overline \alpha$ and at distance 1 from $\alpha$ and 2 from $\beta$, and similarly for that of $\overline \beta$.  By   Lemma \ref{L - loops descriptions}, there must therefore exist a loop $\gamma_1$ joining the extremities of $\overline\alpha$, and similarly a loop $\gamma_2$ joining the extremities of $\overline \beta$.

Therefore, the two other loops must connects vertices pairwise in the set $\{3,6,11,14\}$. There is at most one possibility that is compatible with the model geometries in Figure \ref{F - local geometry collar}. 

\begin{figure}[h]
\includegraphics[width=5cm]{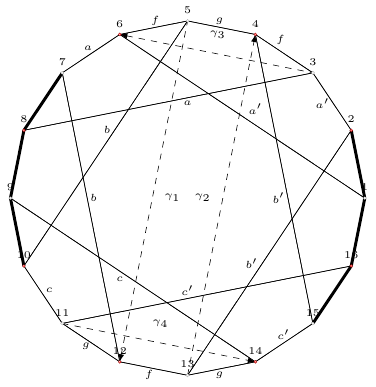}
\caption{The second (rank $2$) group cobordism of Moebius--Kantor type (drawing of the core).}\label{F - Cobordism 2}
\end{figure}

Conversely, it is easy to show that this configuration extends to a model geometry of type $A_\MK$, by completing the two remaining faces $f$ and $g$, as shown in Figure \ref{F - Cobordism 2}.

Finally note that, as in the rank $\frac 3 2$ case:

\begin{lemma} The second group cobordism of Moebius--Kantor type is  self-dual.
\end{lemma}

This concludes the proof of Theorem \ref{T - cobordism 74}.
\end{proof}

\section{Constructions with prescribed type: an example}\label{S - 158}

The aim of the present section is, a) to construct a compact (connected) simplicial complex $X$ of strict type $A_\MK+\tilde A_2$ with two vertices, whose links are respectively  the Moebius--Kantor graph  the incidence graph of the Fano plane and b) to illustrate the constructibility problem concluding the discussion of types in \ts\ref{S - types}.

The Moebius--Kantor graph has 16 vertices and 24 edges, and the double cover $X$ constructed in \ts\ref{S-fake cover}  therefore has 16 faces (and Euler characteristic 2). 

The incidence graph of the Fano plane $P^2\IF_2$, on the other hand, has 14 vertices and 21 edges. Both graphs are transitive and cubic, and there is no a priori obstruction---whether it be combinatorial or of homogeneity---to the existence of a complex $X$ as described in a) above. Such a complex, if it exists, must have 15 faces (and Euler characteristic 2).

It is unclear however that the lack of combinatorial and homogeneity obstructions ensures that  a 2-complex of the given (strict) type actually exists. This is the theme of the type constructibility problem from Section \ref{S - types} mentioned in b).  

Our goal is to construct a complex of strict metric type $A_\MK+\tilde A_2$, which is defined by:
 \begin{enumerate}
 \item[-] the Moebius-Kantor graph and the incidence graph of the Fano plane $P^2\FI_2$, where all edges have length $\pi/3$ \item[-] one equilateral triangle  
 \end{enumerate}

Consider the complex  $X$ defined by the following 15 faces:
\[
X:=[[1,2,3],[1,8,13],[1,12,4],[2,13,10],[2,12,7],[3,7,6],[3,14,7],
\]
\[
[4,4,5],[5,15,14],[5,14,15],[6,9,11],[6,11,9],[8,8,9],[10,13,15],[10,11,12]]
\] 

It is not hard to check that the corresponding group $G$ admits a presentation with 3 generators and 4 relations:
\[
G:=\ip{a,b,c\mid a^2b^{2}=b^2a^{2},\ 
cac^{-1}=b^{3}a^{-2},\
a^2c^{-1}=ba^{-1}c^2b^{-1},\
bab^{-1}=ca^{-1}c^2a^{-1}}.
\]
Observe that $\ZI^2\inj G$. Furthermore, we have:
\[
H_1(X,\ZI) \simeq \ZI\times \ZI/3\ZI\times \ZI/3\ZI.
\]

A direct computation shows that $X$ solves the constructibility problem for groups of type $\frac{15}8$:
\begin{proposition}
The 2-complex $X$ is of strict type $A_\MK+\tilde A_2$.
\end{proposition}

It is obvious that $X$ is not the double cover of a complex. As for the fake double cover in Section \ref{S-fake cover}, one can describe the collar closure explicitly, which consists of  9 faces:
\[
[[1,2,3],[1,8,13],[1,12,4],[2,13,10],\\
\]
\[
[2,12,7],[3,7,6],[3,14,7],[10,13,15],[10,11,12]].
\]
The nerve of this collar, with 6 vertices and 9 edges, is represented in the following figure.
 \[
 \begin{tikzpicture}[baseline=2.7ex]
 \coordinate (A) at (0,0);
 \coordinate (B) at (2,0);
 \coordinate (C) at (2,1);
 \coordinate (D) at (0,1);
 \coordinate (E) at (1,0);
 \coordinate (F) at (1,1);
    \draw (A) -- (B);
    \draw (C) -- (D);
    \draw (E) -- (F);
    \draw (A) to [bend left] (D);
    \draw (B) to [bend left] (C);
    \draw (A) to [bend right] (D);
    \draw (B) to [bend right] (C);
  \end{tikzpicture}
 \]

% \begin{remark}We discovered the complex $X$ in 2009 while deepening the constructions presented in \cite{rd}.\end{remark}

\section{Type preserving surgery}\label{S - type preserving constructions}

Theorem \ref{T - cobordism 74} suggests to construct groups using a surgery similar to the collar surgery in \ts\ref{S-fake cover}, by substituting cobordisms instead of collars. The point of these constructions is that the type of the given group (here type $A_\MK$) is preserved by the surgery.

Here is what a portion of such a group would look like:
\[
\cdots - \stackrel2{\bullet} - \stackrel{\frac 3 2}{\bullet} - \stackrel2{\bullet}-\stackrel2{\bullet} - \stackrel{\frac 3 2}{\bullet} - \stackrel2{\bullet}-\cdots
\]  
where ``$\stackrel2{\bullet}$'' and ``$\stackrel{\frac 3 2}{\bullet}$'' respectively symbolize the first and the second cobordism from Theorem \ref{T - cobordism 74}, and ``$-$'' refers to the collar of type $S$. 

More precisely, the construction starts with two fixed nonisomorphic group cobordisms $X\colon C\to C$ and $Y\colon C\to C$, where $C$ is a collar of type $S$, and the resulting complex of rank \sq\ corresponds to compositions in $\Bord_{\MK}$ as prescribed by the given sequence. In the above example, the complex is $\cdots\circ Y\circ X\circ Y\circ Y\circ X\circ Y\circ\cdots$

In addition, the 2-cover described in \ts\ref{S-double cover} provides filling group cobodisms of type $A_\MK$, whose collars  are also isomorphic to the collars (of type $S$) appearing the cobodisms of Theorem \ref{T - cobordism 74}.

Using these cobordisms as fillings, one can  construct four families of groups of type $A_\MK$:

\begin{enumerate}
\item the segment groups,
\item the circle groups, 
\item the $\IN$-groups, 
\item the $\ZI$-groups, 
\end{enumerate}
which are parametrized, respectively, by
\begin{enumerate}
\item  $\{\frac 3 2,2\}^{\{0,1,\ldots,n\}}$, for $n\geq 0$,
\item $\{\frac 3 2,2\}^{\ZI/n\ZI}$, for $n\geq 2$,
\item $\{\frac 3 2,2\}^{\IN}$,
\item $\{\frac 3 2,2\}^{\IZ}$.
\end{enumerate}
The corresponding groups are finitely presented in case (1) and (2), and infinitely presented in case (3) and (4). 

For example, the group $G_\omega$ in case (1) parametrized by $\omega=(\frac  3 2, 2,2)$ in $\{\frac 3 2,2\}^{\{0,1,2\}}$ is associated with the Moebius--Kantor complex  $X_\omega$ defined by
\[
{\bullet} - \stackrel{\frac 3 2}{\bullet} - \stackrel2{\bullet}-\stackrel2{\bullet} - {\bullet}
\]  
where ${\bullet} - $ and $- {\bullet} $ are filling cobordisms of type $A_\MK$.

One can write down explicit presentations for these groups, at least in cases (1) and (2).
It is clear that:

\begin{proposition}
The groups $G_\omega$ for $\omega\in \{\frac 3 2,2\}^S$, where $S=\{0,1,\ldots,n\},\ \ZI/n\ZI,\ \IN$, or  $\IZ$, admit a free $\frac 23$-transitive action on a CAT(0) Moebius--Kantor complex.
\end{proposition}

We now turn to the space $\Lambda_\MK $ of complexes of type $A_\MK$.  

\begin{lemma}\label{L - at most 2to1}
Let $S$ refer to either $\{0,1,\ldots,n\}$, $\ZI/n\ZI$, $\IN$, or to $\IZ$, and $\Lambda_\MK $ denote the space of Moebius--Kantor complexes. 
The map 
\begin{align*}
\rho\colon\{\frac 3 2,2\}^S&\to \ \ \Lambda_\MK \\
\omega &\mapsto (X_\omega,0)
\end{align*}
where $(X_\omega,0)$ means that the complex $X_\omega$ pointed at the unique vertex in the cobordism over $0\in S$, is 
\begin{itemize}
\item injective if $S=\{0,1,2,\cdots,n\}$ or $S=\IN$, and,
\item at most 2-to-1 if $S=\ZI/n\ZI$ or $S=\ZI$.
\end{itemize}
\end{lemma}

Note that $\rho$ is continuous, and has as image a compact subset of $\Lambda_\MK $.

\begin{proof}
Assume first that $S=\{0,1,2,\cdots,n\}$ or $S=\IN$. Let  $\omega, \omega'$ be two elements of $\{\frac 3 2,2\}^S$. If 
\[
\p\colon (X_\omega,0)\to (X_{\omega'},0)
\] 
is an isomorphism, then $\p$ sends the cobordism over $t\in S$ in $X_\omega$ to the cobordism over $t$ in $X_{\omega'}$. Indeed this is so by definition for $0\in S$, and the result follows by induction. 
Therefore, for every $s\in S$ we have $\omega(s) = \omega'(s)$, so $\omega=\omega'$. 

If $S=\ZI/n\ZI$ or $S=\ZI$, induction (namely, the 1-dimensional nature of $S$) shows that either $\omega(k) = \omega'(k)$ or $\omega(k) = \omega'(-k)$. Therefore the  map $\rho$ is at most 2-to-1.  
\end{proof}

We will prove: 

\begin{theorem}\label{T - 74 diffuse}
The space $\Lambda_\MK $ supports a diffuse invariant probability measure. 
\end{theorem}

\begin{proof}
Consider the semi-direct product $G= \ZI\rtimes \ZI/2\ZI$
acting on the Cantor set $X:=\{\frac 3 2,2\}^\ZI$ as follows
\[
s\cdot (\omega_k)_{k\in \ZI} = (\omega_{s^{-1}k})_{k\in \ZI}.
\]
Endow $X$ with the Bernoulli measure $\mu:=(\frac 1 2 \delta_{3/2}+\frac 1 2\delta_{2})^{\otimes \ZI}$.
It is well--known and easy to prove that:
\begin{lemma}\label{L - essentially free}
The action $G\acts (X,\mu)$ is an essentially free pmp action.
\end{lemma}

\begin{proof}
The fact that $\mu$ is invariant is clear. Let us explain why the action is essentially free. Let $e\neq s\in G$.  Then $s\omega=\omega$ if and only if  $\omega$ is constant on the $s$ orbit in $\ZI$. Since $s\neq e$, we can find infinitely many pairwise disjoint pairs  of integers (which depend on $s$), say $(k_i,l_i)_{i\in \IN}$, such that  $\omega_{k_i}=\omega_{l_i}$ for every fixed point $\omega$ of $s$. This gives infinitely many constraints on the coordinates of the elements $\omega$ in the fixed point set $X^s$, and since $\mu$ is a product measure, it follows that $\mu(X^s)=0$.
\end{proof}

\begin{remark}
The action $G\acts (X,\mu)$ is sometimes called a generalized Bernoulli shift, and is useful to give examples of operator algebras.%  (see e.g.\  \cite{popa2008strong} and references therein).
\end{remark}
 
The map 
\begin{align*}
\rho\colon \{\frac 3 2,2\}^\ZI&\to \ \ \Lambda_\MK \\
\omega &\mapsto (X_\omega,0)
\end{align*}
is an orbit map, namely $x\sim_G y\ssi \rho(x)\sim_{R_A}\rho(y)$. 

Indeed,  if $s\in G$ and $\omega'=s\omega$, then by construction
$(X_{\omega'},0)=(X_\omega,s^{-1}(0))$, so  $(X_{\omega'},0)\sim_{R_A}(X_\omega,0)$. Conversely if $(X_{\omega},0)\sim_{R_A}(X_{\omega'},0)$ then let $\p\colon X_\omega\to X_{\omega'}$ be an isometry. Since the geometry encodes the sequences $\omega$ and $\omega'$ in the  corbodisms at vertices, in an orderly fashion, there exists a $k_0\in \ZI$ such that that either $\omega(k) = \omega'(k_0+k)$ or $\omega(k) = \omega'(k_0-k)$ for all $k\in \ZI$. In both cases we can find $s\in G$ such that $\omega'=s\omega$.

\begin{lemma}\label{L - 2to1}
The map $\rho$ is essentially 2-to-1.
\end{lemma}
 
 \begin{proof}
 By essentially 2-to-1 we mean that there is a measurable set $X'\subset X$ of full measure $\mu(X')=1$ and such that $\rho_{|X'}$ is 2-to-1. By Lemma \ref{L - at most 2to1}, $\rho$ is at most 2-to-1. If $\rho(\omega)=\rho(\omega')$ and $\omega\neq \omega'$ then $\omega(k) = \omega'(-k)$ for all  $k\in \ZI$, that is, $\omega'=s\omega$ where $s\in G$ is the symmetry around 0. Therefore the restriction of $\rho$ to $X':=X\setminus X^s$ is a 2-to-1 map.  By Lemma \ref{L - essentially free}, $\mu(X^s)=0$.
 \end{proof}

\begin{lemma}\label{L - lemma push forward}
The measure $\rho_*(\mu)$ is invariant (and diffuse).
\end{lemma}

This lemma will be a consequence of Lemma \ref{L - 2to1} and Lemma \ref{L - quotient measure} below, and it concludes the proof of Theorem \ref{T - 74 diffuse}.  \end{proof}

This suggests the following definition.
 
\begin{definition}\label{D - structure group}
Let $A$ be a type. A \emph{structure group} for  $A$ is a countable group $G$ for which there exists an integer $n\geq 1$, an essentially free  pmp action $G\acts X$ on standard probability space $X$, and an essentially $n$-to-1 Borel map
\[
\rho \colon X\to \Lambda_A
\]  
such that $x\sim_G y\ssi \rho(x)\sim_{R_A} \rho(y)$ for every $x,y\in X$.
\end{definition} 

It may be useful to also consider the case of nonsingular actions (instead of pmp). 
The map $\rho$ allows to push-forward (diffuse, invariant) measures on $X$  to (diffuse, invariant) measures to the space $\Lambda_A$:

\begin{lemma}\label{L - quotient measure}
Let $\rho\colon X\to Y$ be a  Borel map between standard probability spaces.  Assume that
\begin{enumerate}[a)]
\item $R\subset X\times X$ and $S\subset Y\times Y$ are standard Borel equivalence relations with countable classes
\item $\rho$ is an orbit map, in the sense that $x\sim y\ssi \rho(x)\sim \rho(y)$ for all $x,y\in X$.
\item $\mu$ is an $R$-invariant  probability measure on $X$
\item $n\geq 1$ is an integer such that $\rho$ is essentially $n$-to-1 with respect to $\mu$ (there exists  a measurable set $X'\subset X$ with $\mu(X')=1$ such that $\rho_{|X'}$ is $n$-to-1)
\end{enumerate}
Then the measure $\rho_*(\mu)$ is an $S$-invariant measure on $Y$.   
\end{lemma}

This lemma does not seem to belong to the literature on orbit equivalence (see e.g.\ \cite{kechris2004topics} or \cite{hjorth2000classification}), so we give a proof.  

\begin{proof}
Consider the map 
\[
\pi:=\rho\times \rho\colon X\times X\to Y\times Y
\] 
and its restriction $\pi\colon R\to S$. It is clear that
\[
\pi\circ \sigma= \sigma\circ \pi
\]
where $\sigma$ denotes the flip map $(x,y)\to (y,x)$ on both $R$ and $S$. The measure $\mu$ is invariant if and only if 
\[
\sigma_*\nu=\nu
\]
where 
\[
\nu(A):=\int_X |A\cap \{x\}\times X|\d\mu(x)
\]
is the right-counting measure on $R$ (see \cite[Theorem 2]{feldman1977ergodic}). 

Since
\[
\sigma_*(\pi_*\nu)=(\sigma\circ \pi)_*\nu=(\pi\circ \sigma)_*\nu = \pi_*( \sigma_*\nu) = \pi_*\nu,
\]
the lemma follows from the fact that $\pi_*\nu$ is a multiple of the right-counting measure $\nu_\rho$ on $S$ associated with $\rho_*\mu$: 
\begin{align*}
\pi_*\nu(A)&=\int_X|\pi^{-1}(A)\cap \{x\}\times X|\d\mu(x)\\
&=\int_X|\{(x,y)\in R, (\rho(x),\rho(y))\in A\}|\d\mu(x)\\
&=\int_X|\{y\in X, (\rho(x),\rho(y))\in A\}|\d\mu(x)\\
&=\int_Y|\{y\in X, (z,\rho(y))\in A\}|\d\rho_*\mu(z)\\
&=n\int_Y|\{t\in Y, (z,t)\in A\}|\d\rho_*\mu(z)\\
&=n\nu_\rho(A)
\end{align*}
for every Borel set $A\subset S$.
\end{proof}

\begin{remark}
There is a pointwise (less functorial) proof of Lemma \ref{L - quotient measure}. Let $\psi\subset S\times S$ be a partial automorphism of $S$. Then $\pi^{-1}(\psi)$ is a subset of $R$ whose (left and right) fibers over $X$ are either empty, or they have cardinality $n$. Using the Lusin selection theorem (for Borel maps with countable fibers), one can find partitions:
\[
\rho^{-1}(\dom \psi) = \bigsqcup_{i=1}^n A_i,\ \ \ \rho^{-1}(\ran \psi) = \bigsqcup_{j=1}^n B_j,
\]  
and partial isomorphisms $\p_{i,j}\colon A_i\to B_j$ of $R$ such that 
\[
\pi^{-1}(\psi)=\bigsqcup_{i,j=1} ^n \p_{i,j}.
\]
Since $\mu$ is invariant, we have $\mu(A_i)=\mu(B_j)$ for all $i,j=1\ldots, n$, and 
\[
\rho_*\mu(\dom \psi)=n\mu(A_1)=n\mu(B_1)=\rho_*\mu(\ran \psi)
\]
which shows that $\rho_*\mu$ is $S$-invariant. 
\end{remark}

Theorem \ref{T - 74 diffuse} is, therefore, a consequence of:

\begin{theorem}\label{T - structure group}
The dihedral group $D_\infty$ is  a structure group for the type $A_\MK $. 
 \end{theorem}
 
This theorem shows that not only surgery for the type $A_\MK $ is  possible, but that it can be done with great freedom.  The  questions raised in \cite{Esurv} for the type $\tilde A_2$ remain open.

\end{document}